\documentclass{amsart}

\usepackage{amsmath}
\usepackage{amsfonts}
\usepackage{latexsym}
\usepackage{mathrsfs}
\usepackage{array}
\usepackage{amscd}
\usepackage[mathcal]{euscript}
\usepackage{amssymb}
\usepackage{pstricks}
\usepackage{amsrefs} 

\newtheorem{maintheorem}{Main Theorem}

\newtheorem{mainapplication}{Main Application}

\newtheorem{theorem}{Theorem}[section]
\newtheorem{lemma}[theorem]{Lemma}
\newtheorem{proposition}[theorem]{Proposition}

\theoremstyle{definition}
\newtheorem{definition}[theorem]{Definition}
\newtheorem{example}[theorem]{Example}

\theoremstyle{remark}
\newtheorem{remark}[theorem]{Remark}

\numberwithin{equation}{section}

\input xy 
\xyoption{all} 
\input{boxedeps} 
\HideDisplacementBoxes
\SetRokickiEPSFSpecial  
\newcommand{\bB}{\mathbb{B}} 

\newcommand{\cB}{\mathcal{B}} 
\newcommand{\cC}{\mathcal{C}} 
 
\newcommand{\cF}{\mathcal{F}} 
\newcommand{\cH}{\mathcal{H}} 
\newcommand{\cK}{\mathcal{K}} 
 
\newcommand{\cM}{\mathcal{M}} 
\newcommand{\cP}{\mathcal{P}} 
\newcommand{\cS}{\mathcal{S}} 
\newcommand{\cT}{\mathcal{T}}

\newcommand{\ra}{\rightarrow}

\newcommand{\ob}{\overline} 

\newcommand{\bx}{{\bf x}}
\newcommand{\by}{{\bf y}}
\newcommand{\bz}{{\bf z}}
\newcommand{\cpr}{\cC\cP_R}
\newcommand{\bundr}{\cB und_R}
\newcommand{\bundrB}[1]{\cB und_R({#1})}
\newcommand{\spf}[1]{\cS_{#1}(P,\cF)}
\newcommand{\hpc}[1]{H_{#1}(P,\cF)}
\newcommand{\cpc}[1]{\cC_{#1}(P,\cF)}

\newcommand{\kh}[2]{K\!H_{{#1},{#2}}}
\newcommand{\khoriginal}[2]{K\!H^{{#1},{#2}}}
\newcommand{\obkh}[2]{\ob{K\!H}_{{#1},{#2}}}
\newcommand{\rmod}{\cM\text{od}_R}
\newcommand{\grrmod}{Gr\cM\text{od}_R}
\newcommand{\chr}{\text{Ch}_R}

\newcommand{\obc}[2]{\overline{\cC}_{{#1},{#2}}}
\newcommand{\obcoriginal}[2]{\overline{\cC}^{{#1},{#2}}}
\newcommand{\obcK}{\overline{\cK}}
\newcommand{\oE}{\overline{E}}

\DeclareMathAlphabet{\ams}{U}{msb}{m}{n}
\DeclareMathAlphabet{\goth}{U}{euf}{m}{n}

\def\sl{\text{SL}}

\def\id{\text{id}}

\def\rk{\text{rk}\,}

\def\ker{\text{ker}\,}
\def\ov{\overline}

\def\FF{\mathcal F}
\def\EE{\mathcal E}

\def\CC{\mathcal C}

\def\TT{\mathcal T}

\def\SS{\mathcal S}

\def\gS{\goth{S}}

\def\aa{\alpha}

\def\ss{\sigma}

\def\Z{\ams{Z}}

\def\B{\ams{B}}

\def\x{\mathbf{x}}
\def\y{\mathbf{y}}
\def\z{\mathbf{z}}
\def\zhat{\hat{\mathbf{z}}}
\def\0{\mathbf{0}}
\def\quo{/\kern -.45em\sim}
\newpsobject{showgrid}{psgrid}{subgriddiv=1,griddots=10,gridlabels=6pt,gridcolor=red}
\def\Langle{\langle\kern -2pt\langle}
\def\Rangle{\rangle\kern -1.9pt\rangle}
\newgray{lightergray}{.85}
\begin{document}

\title[Bundles and Khovanov homology]{Bundles of coloured posets and a Leray-Serre spectral sequence for Khovanov homology}

\author{Brent Everitt}
\address{Department of Mathematics, University of York, York YO10 5DD, England}
\email{brent.everitt@york.ac.uk}
\thanks{The first author thanks
Finnur Larusson for many useful and stimulating discussions. He 
is also
grateful to the Alpine Mathematical Institute, Haute-Savoie, France,
and to the Institute for Geometry and its Applications, University of
Adelaide, Australia.}

\author{Paul Turner}
\address{D\'epartement de math\'ematiques, 
Universit\'e de Fribourg, CH-1700 Fribourg, Switzerland.}
\email{prt.maths@gmail.com}
\thanks{The second author was partially supported by the Swiss National Science Foundation projects no. 200020-113199 and no. 200020-121506.}







\begin{abstract}
The decorated hypercube found in the construction of Khovanov homology
for links is an example of a Boolean lattice equipped with a presheaf
of modules. One can place this in a wider setting as an example of a coloured
poset, that is to say a poset with a unique maximal element equipped
with a presheaf of modules. In this paper we initiate the study of a
bundle theory for coloured posets, producing for a certain class of
base posets a Leray-Serre type spectral sequence. We then show how
this theory finds application in Khovanov homology by producing a new
spectral sequence converging to the Khovanov homology of a given link.
\end{abstract}


\maketitle

\section*{Introduction}
Our motivation for studying coloured posets comes from a desire to
understand better some of the structures arising in the 
Khovanov homology of links \cite{Khovanov00, Bar-Natan02}. One can ask if Khovanov's ``cube'' construction
can be placed in a broader context, hoping that this
new perspective allows one to lever some advantage. In
\cite{Everitt-Turner} we showed that this is possible using the
framework of what we call coloured posets. This is something of a
half-way house between the specifics of ``cubes'' and the full
generality of presheaves of modules over small categories. We believe
this is an appropriate place to study Khovanov homology as it provides
enough generality to be useful, without losing sight of the starting
point. The idea being that one can then bring the algebraic topology
of coloured posets into the game.

One direction of interest is a bundle theory for coloured posets which
we initiate in this paper. Given a poset $B$ with unique maximal
element we define bundles of coloured posets over $B$. Roughly, we
wish to capture the idea that a coloured poset may be decomposed as a
family of coloured posets parametrized by another poset, much as a
fibre bundle is a family of spaces parametrized by some base space.
A bundle of coloured posets has a total
coloured poset $(E,\cF)$, and our central interest is in computing its
homology.  Our main result in this direction, proved as Theorem
\ref{thm:main} in \S\ref{sec:main}, is a Leray-Serre style spectral
sequence for bundles where the base has a certain technical property that
we call special admissibility:

\begin{maintheorem}
Let $\xi:B\rightarrow\cpr$ be a bundle of coloured posets with $B$ 
finite and specially admissible, and 
$(E,\cF)$ the associated total coloured poset. Then there is a spectral sequence
\[
E^2_{p,q} = H_p(B, \cH^{\text{{\rm fib}}}_q (\xi)) \Longrightarrow H_*(E,\cF).
\]
\end{maintheorem}

It turns out that a number of naturally ocurring posets are
specially admissible, and
importantly for applications to
knot homology theories, this class includes the Boolean lattices.
Given a knot diagram $D$ together with $k$
distinguished crossings $c_1, \ldots , c_k$, one can obtain a new complex in
the spirit of Khovanov: resolving each of
the remaining $\ell$ crossings gives another link diagram, and the
$\ell$ diagrams that result
can be placed as the vertices of a
hyper-cube (a Boolean lattice).  Taking the unnormalised Khovanov
homology of these diagrams and using induced maps along edges one
obtains, by the construction of Khovanov, a triply graded complex
whose homology we denote $ \ob{K\!H}_{*,*,*}(D; c_1, \ldots , c_k)$.

\begin{mainapplication}
Let $D$ be a link diagram and let $c_1, \ldots , c_k$ be a subset of
the crossings of $D$. 
For each $i$, there is a spectral sequence
\[
E^2_{p,q} = \ob{K\!H}_{p,q,i}(D; c_1, \ldots , c_k)  \Longrightarrow \obkh {p+q} i (D).
\]
The $r$-th differential in the spectral sequence has bidegree $(-r,r-1)$.
\end{mainapplication}

The paper unfolds as follows. In the next section we define
bundles of coloured posets, describe the associated
total coloured poset, and give a number of examples. We also 
recollect facts from \cite{Everitt-Turner} about coloured posets and their
homology. In \S\ref{section300} we start with a bundle of
coloured posets and introduce a certain bi-complex which (by standard
methods) gives rise to a spectral sequence converging to the homology
of the total complex. Our task later will be to recognise
this as the homology of the total coloured poset of the bundle we started
with. In fact we can only do this for a restricted class of bases, the
specially admissible ones, and 
we introduce these in \S\ref{section400}. We give a number of
examples, including the Boolean lattices. Following this,
\S\ref{sec:les} introduces the technical tools needed to
prove the main theorem, namely certain
long exact sequences. In \S\ref{sec:main} we state and prove the
main theorem. Finally in \S\ref{sec:appl} we turn to Khovanov
homology and construct a new spectral sequence which for a given link has a
potentially computable $E^2$-page, and which converges to the Khovanov homology
of the link.

\section{Bundles of coloured posets}
\label{sec:bund}

Recall from \cite{Everitt-Turner} that a coloured poset is a pair
$(P,\cF)$ consisting of a poset $P$
having a unique maximal element $1_P$, and a covariant functor $\cF
\colon P \ra \rmod$, called the {\em colouring}. This may be regarded
as a representation of $P$ or as a presheaf of modules over $P$ according
to taste. A {\em morphism } of coloured posets $(P_1,\cF_1) \ra
(P_2,\cF_2)$ is a pair $(f,\tau)$ where $f\colon P_1 \ra P_2$ is a map
of posets, and $\tau=\{\tau_x\}_{x\in P_1}$ is a collection of
$R$-module homomorphisms $\tau_x\colon \cF_1(x) \ra \cF_2(f(x))$.
This pair must satisfy (i) $f(x) = 1_{P_2}$ if and only if
$x=1_{P_1}$, and (ii) for all $x\leq y$ in $P_1$, the following
diagram commutes
$$
\begin{pspicture}(12.5,3.25)
\rput(-1.05,0){
\rput(6,0.5){
\psframe[linecolor=lightergray,fillstyle=solid,fillcolor=lightergray](-1.05,-.5)(1.05,2.5)
\rput(0,2.75){$P_1$}
\psframe[linecolor=lightergray,fillstyle=solid,fillcolor=lightergray](1.95,-.5)(4.05,2.5)
\rput(3,2.75){$P_2$}
\rput(0,0){
\rput(0,0){$\cF_1(x)$}
\psline[linewidth=.3mm]{->}(0,.3)(0,1.7)
\psframe[linecolor=lightergray,fillstyle=solid,fillcolor=lightergray](-1.,.75)(1.,1.25)
\rput(0,1){$\scriptstyle{\cF_1(x\leq y)}$}
\rput(0,2){$\cF_1(y)$}
}
\rput(3,0){
\rput(0,0){$\cF_2(f(x))$}
\psline[linewidth=.3mm]{->}(0,.3)(0,1.7)
\psframe[linecolor=lightergray,fillstyle=solid,fillcolor=lightergray](-1.05,.75)(1.05,1.25)
\rput(0,1){$\scriptstyle{\cF_2(f(x)\leq f(y))}$}
\rput(0,2){$\cF_2(f(y))$}
}
\psline[linewidth=.3mm]{->}(.6,0)(2.15,0)
\rput(1.35,.2){$\scriptstyle{\tau_x}$}
\psline[linewidth=.3mm]{->}(.6,2)(2.15,2)
\rput(1.35,2.2){$\scriptstyle{\tau_y}$}
}}
\end{pspicture}
$$
Coloured posets and morphisms between them form the category $\cpr$.

We also recall from \cite{Everitt-Turner} our poset convention that
if $P$ has a unique minimal element $0_p$, then $0_p<1_p$. In particular
the poset with a single element has a $1$ but not a $0$. If $x<y$ in $P$ and
there is no $z$ with $x<z<y$, then one says that $y$ covers $x$ and writes $x\prec y$.

Throughout this paper, if $x\in P$ we write $P(x)$ for the
interval 
$$P(x):=\{y\in P\,|\,y\leq x\}
$$ 
and $\ov{P}(x)$ for it's complement
$$
\ov{P}(x):=\{y\in P\,|\,y\not\leq x\}.
$$ 
From now on we will just say ``poset with 1'' to refer to a
poset with unique maximal element.

The idea of a bundle is that it consists of a collection of coloured posets 
(the fibres) parametrized by another poset (the base) that encodes instructions
for gluing the fibres together.

\begin{definition}\label{sec:bund:definition100}
Let $B$ be a poset with unique maximal element $1_B$. 
A {\em bundle of coloured posets with base} $B$ is a (covariant) functor $\xi \colon B \ra \cpr$.
\end{definition}

Thus, to each $x\in B$ there is an associated coloured poset $ \xi(x)=
(E_x,\cF_x)$, the {\em fibre} over $x$. Whenever
$x\leq z$ there is an associated morphism of coloured posets 
$\xi(x\leq z) = (f_x^z,\tau_x^z)$ from $\xi(x)$ to $\xi(z)$. 
There is a projection map of (uncoloured) 
posets $\pi\colon E \ra B$ defined by  
$\pi(y)=x$ if and only if $y\in P_x$.
Definition \ref{sec:bund:definition100} is akin to
the interpretation of a vector bundle (with connection) as a functor
from the path space of the base to the category of vector
bundles. What is missing from this definition is the analogue of the
total space, which is somewhat hidden but which can be constructed as
follows.

\begin{definition}\label{sec:bund:definition200}
Let $B$ be a poset with 1 and let $\xi \colon B \ra \cpr$ be a bundle
of coloured posets. The associated {\em total coloured poset} is the
coloured poset $(E,\cF)$ defined by: 
\begin{itemize}
\item As a set $E = \bigcup_{x\in B} E_x$, the union of the fibres,  
and the partial order on $E$ is defined as follows:
\begin{itemize}
\item if $y,y^\prime\in E_x$, so $y$ and $y^\prime$ are in the same
fibre, then $y\leq y^\prime$ iff $y\leq y^\prime$ in $E_x$,
\item if $y\in E_x$ and $y^\prime\in E_{x^\prime}$ with $x\neq
 x^\prime$, then $y\leq y^\prime$ iff $x\leq x^\prime$ in $B$ and
 $f_x^{x'}(y)\leq y^\prime$ in $E_{x^\prime}$.
\end{itemize}
\item The colouring $\cF\colon E \ra \cpr$ is defined
on an object $y$ in fibre $E_x$ by setting $\cF(y) = \cF_x(y)$,
and on a morphism $y\leq y^\prime$ as follows:
\begin{itemize}
\item  if $y,y^\prime\in E_x$ then $\cF(y\leq y^\prime) = \cF_x(y\leq y^\prime)$, 
\item if $y\in E_x$ and $y^\prime\in E_{x^\prime}$ with $x\neq
 x^\prime$, then $$\cF(y\leq y^\prime) = \cF_{x^\prime}(f_x^{x'}(y)
 \leq  y^\prime)\circ (\tau_x^{x^\prime})_y.$$
\end{itemize}
\end{itemize}
\end{definition}

In the last line of the definition we are using the morphism of coloured posets 
$\xi (x\leq x^\prime) = (f_x^{x^\prime}, \tau_x^{x^\prime})$ and recalling
that $ (\tau_x^{x^\prime})_y$ is a map $\cF_x(y)\rightarrow\cF_{x'}(f_x^{x'}(y)) $. 

\begin{lemma}
The pair $(E,\cF)$ of Definition \ref{sec:bund:definition200} is a coloured poset.
\end{lemma}

\begin{proof}
It is easily checked that $E$ is a poset. For the colouring we must verify that 
composition behaves as it should. There are four cases to consider of which we 
treat one in detail, the others being much simpler.

Let $y\in E_x$,  $y'\in E_{x'}$ and  $y''\in E_{x''}$ where $x\neq x'\neq x''$. 
Suppose that $y\leq y'$ and $y'\leq y''$. We must show that 
$\cF(y'\leq y'')\circ \cF(y\leq y') = \cF(y\leq y'')$. Consider the following diagram
$$
\begin{pspicture}(12.5,6)
\rput(-1.7,0){
\rput(-.5,0){
\psframe[linecolor=lightergray,fillstyle=solid,fillcolor=lightergray](2.7,0)(4.2,2.5)
\rput(2.45,2.3){$E_x$}
}
\psframe[linecolor=lightergray,fillstyle=solid,fillcolor=lightergray](5.5,0)(7.7,4.5)
\rput(5.2,4.3){$E_{x'}$}
\rput(.5,0){
\psframe[linecolor=lightergray,fillstyle=solid,fillcolor=lightergray](8.5,0)(11.5,6)
\rput(8.1,5.8){$E_{x''}$}
}
\rput(3.25,1.5){
\rput(-.3,0){$\cF_{x}(y)$}
\psline[linewidth=.3mm]{->}(.3,0)(2.4,0)
\rput(1.45,.3){$\scriptstyle{(\tau_{x}^{x'})_{y}}$}
\rput(3.4,0){
\rput(0,0){
\rput(0,0){$\cF_{x'}(f_x^{x'}(y))$}
\psline[linewidth=.3mm]{->}(0,.3)(0,1.7)
\psframe[linecolor=lightergray,fillstyle=solid,fillcolor=lightergray](-.85,.75)(.85,1.25)
\rput(0,1){$\scriptstyle{\cF_{x'}(f_x^{x'}(y)\leq y')}$}
\rput(0,2){$\cF_{x'}(y')$}
}
\psline[linewidth=.3mm]{->}(.95,0)(2.8,0)
\rput(1.8,.3){$\scriptstyle{(\tau_{x'}^{x''})_{f_x^{x'}(y)}}$}
\psline[linewidth=.3mm]{->}(.8,2)(2.6,2)
\rput(1.65,2.3){$\scriptstyle{(\tau_{x'}^{x''})_{y'}}$}
\rput(3.9,0){
\rput(0,0){$\cF_{x''}(f_x^{x''}(y))$}
\psline[linewidth=.3mm]{->}(0,.3)(0,1.7)
\psframe[linecolor=lightergray,fillstyle=solid,fillcolor=lightergray](-1.05,.75)(1.05,1.25)
\rput(0,1){$\scriptstyle{\cF_{x''}(f_x^{x''}(y)\leq f_{x'}^{x''}(y'))}$}
\rput(0,2){$\cF_{x''}(f_{x'}^{x''}(y'))$}
\psline[linewidth=.3mm]{->}(0,2.3)(0,3.7)
\psframe[linecolor=lightergray,fillstyle=solid,fillcolor=lightergray](-1,2.75)(1,3.25)
\rput(0,3){$\scriptstyle{\cF_{x''}(f_{x'}^{x''}(y')\leq y'')}$}
\rput(0,4){$\cF_{x''}(y'')$}
\psline[linewidth=.3mm,linearc=.3]{->}(1.1,0)(2.5,0)(2.5,4)(.75,4)
\rput*(2.7,2){$\scriptstyle{\cF_{x''}(f_x^{x''}(y)\leq y'')}$}
}
}
\psline[linewidth=.3mm,linearc=.3]{->}(-.3,-.3)(-.3,-1)(7.3,-1)(7.3,-.3)
\rput(3.4,-.7){$\scriptstyle{(\tau_{x}^{x''})_{y}}$}
}}
\end{pspicture}
$$
The triangle at bottom commutes since $\xi$ is a functor, the 
triangle at the right commutes because $\cF_{x''}$ is a functor and the
square commutes by the naturality of $ \tau_{x'}^{x''}$. Thus the
entire diagram commutes. Going up the steps
gives $\cF(y'\leq y'')\circ \cF(y\leq y')$ while 
following the two arrows with bends gives $\cF(y\leq y'')$. 
\end{proof}

\begin{example}
Figure \ref{figure100} gives an example of a bundle over a Boolean lattice of rank two,
with the colouring left off.
\begin{figure}
  \centering
\begin{pspicture}(12.5,5)
\rput(-1.25,0){
\rput(7.5,2.5){\BoxedEPSF{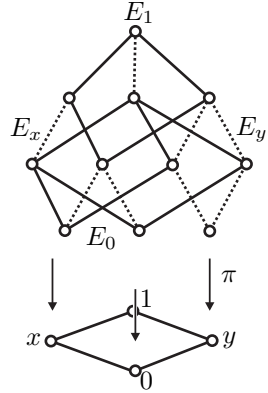 scaled 670}}
\rput(7.6,.1){$0$}\rput(6.1,.65){$x$}
\rput(8.7,.65){$y$}\rput(7.6,1.2){$1$}
\rput(7,2){$E_0$}\rput(6,3.45){$E_x$}
\rput(9,3.45){$E_y$}\rput(7.5,5){$E_1$}
\rput(8.7,1.5){$\pi$}
}
\end{pspicture}
  \caption{Example of a bundle with the colouring omitted. 
}
  \label{figure100}
\end{figure}
\end{example}

\begin{example}
Let $B$ be a poset with 1 and let $(P,\cF)$ be a coloured poset. 
The {\em product bundle} with base $B$ and fibre $(P,\cF)$, denoted $B \times (P,\cF)$, 
is defined by the functor $\xi \colon B \ra \cpr$ where $\xi(x) =(P,\cF)$ and 
$\xi(x_1 \leq  x_2)= \id$. 
\end{example}

\begin{example}
Let $B$ be the Boolean lattice of rank $1$, ie: the poset with two elements 
$0,1$ and $0<1$, and let $\xi$ be a bundle over $B$. 
The associated total coloured poset $(E,\FF)$ 
is the result of gluing the coloured posets $(P_0,\FF_0)$ and 
$(P_1,\FF_1)$ along the morphism 
$$
(f_0^1,\tau_0^1):(P_0,\FF_0)\rightarrow(P_1,\FF_1),
$$
as defined in \cite{Everitt-Turner}*{\S 1}.
\end{example}

\begin{example}\label{ex:boolean}

Let $X$ be a set, $\B(X)$ the Boolean lattice of all subsets of $X$, and 
$\cF:\B(X)\rightarrow\rmod$ a colouring. If $A\in\B(X)$ then $(\B(X),\cF)$ decomposes
as a bundle with base $\B(X\setminus A)$ and each fibre a copy of $\B(A)$: a 
$B\in\B(X)$ lies in the fibre over $B\cap(X\setminus A)$ at the element 
$B\cap A\in\B(A)$, and is coloured $\cF(B)$. This is just a decomposition
of $\B(X)$ as a product of posets $\B(X\setminus A)\times\B(A)$.
If $Y\in\B(X\setminus A)$ then $E_Y=\{Y\cup Z\,|\,Z\in\B(A)\}$ and
$Y\cup Z\mapsto Y'\cup Z$ is the poset map $f_{Y}^{Y'}:E_Y\rightarrow E_{Y'}$
when $Y\leq Y'\in\B(X\setminus A)$, with $(\tau_{Y}^{Y'})_{Y\cup Z}$ equal to the map
$\cF(Y\cup Z\leq Y'\cup Z)$.
\end{example}

\begin{example}\label{ex:int}
If $B$ is a poset with 1, and $\xi$ is a bundle over $B$, then any
sub-poset $A\subseteq B$ having a unique maximal element, has an
associated bundle obtained by restricting $\xi$. This is given by the
composition of functors $A\ra B \ra \cpr$ and will be denoted
$\xi|_A$.  Note that the maximal element of $A$ need not be the
maximal element in $B$.  Two instances of particular importance arise
when $\xi$ is a bundle over $B$ with total coloured poset $(E,\FF)$,
and for $x\in B$ we consider the subposets the interval $B(x)$ and its
complement $\ov{B}(x)$.  The interval $B(x)$ has maximal element $x$
and the complement has maximal element the $1$ of $B$.  We will denote
the total coloured posets of the restricted bundles $\xi |_{B(x)}$ and
$\xi |_{\ov{B}(x)}$ by $(E(x),\FF)$ and $(\ov{E}(x),\FF)$.
\end{example}

Bundles of coloured posets form a category $\bundr$ as we now describe.

\begin{definition}
Let $\xi\colon B \ra \cpr$ and $\xi^\prime\colon B^\prime \ra \cpr$ be bundles 
of coloured posets over $B$ and
$B^\prime$ respectively. A {\em morphism} $(g,\eta)\colon \xi\ra\xi^\prime$ is a map
of (uncoloured) posets $g\colon B\ra B^\prime$ satisfying $g(x) = 1$
iff $x=1$, together with a natural transformation $\eta$ from $\xi$ to
$\xi^\prime\circ g$. 
\end{definition}

By fixing the base poset $B$ and restricting morphisms to those with $g=\id$, 
one obtains a sub-category of bundles over $B$ denoted $\bundrB B$.

Associating the total coloured poset to a bundle is
functorial: given a bundle $\xi\colon B \ra \cpr$
let $\EE(\xi)=(E,\FF)$ be the associated total coloured
poset. For a morphism $(g,\eta)\colon \xi\ra\xi^\prime$ between
bundles $\xi\colon B \ra \cpr$ and $\xi^\prime\colon B^\prime\ra\cpr$, 
we define a morphism of coloured posets $\EE(g,\eta):\EE(\xi)\ra\EE(\xi')$ 
in the following way. Recall that
we must define a poset map $f\colon E \ra E^\prime$ and a collection
$\{\sigma_y\colon \FF(y) \ra \FF^\prime(f(y))\}$,
and that for $x\in B$, the natural
transformation $\eta$ gives a morphism of coloured posets,
$$
\eta(x):=(f_{x}^{g(x)},\tau_{x}^{g(x)}):\xi(x)\rightarrow\xi'g(x),
$$
where 
$\tau_{x}^{g(x)}$ is the collection of morphisms $\{(\tau_{x}^{g(x)})_y\}_{y\in\xi(x)}$.
Define  $f:E\rightarrow E'$ by $f(y)=f_{x}^{g(x)}(y)$ for $y\in \xi(x)$, and 
$\ss_y= (\tau_{x}^{g(x)})_y$.

If $x\leq z\in B$ we have the commuting diagram,
$$
\xymatrix{
\xi(z) \ar[r]^{\eta(z)}  &\xi'g(z) \\
\xi(x) \ar[r]^{\eta(x)} \ar[u]^{(f_x^z,\tau_x^z)} &\xi'g(x) 
\ar[u]_{(f_{g(x)}^{g(z)},\tau_{g(x)}^{g(z)})}
}
$$
courtesy of the natural transformation $\eta$, from which one can deduce that 
$(f,\ss)$ is a morphism of coloured posets. Moreover, it preserves fibres.
Thus given a morphism  $(g,\eta)\colon \xi\ra\xi^\prime$  
we have assigned a morphism of coloured posets $\EE(g,\eta)= (f,\sigma)$. 
Since composition behaves well we have,

\begin{proposition}
$\EE:\bundr\rightarrow\cpr$ is a functor. 
\end{proposition}

There is a homology of coloured posets and our main interest
in this paper is to compute it for the total coloured poset
of a given bundle. We now briefly recall the relevant definitions-for
more details see \cite{Everitt-Turner}.  Let $(P,\cF)$ be a coloured
poset and define a chain complex $\cS_*(P,\cF)$ by 
letting $\cS_k(P,\cF) = 0$ for $k<0$; for $k>0$ let
\begin{equation}
\cS_k(P,\cF) =\kern-2mm\bigoplus_{\substack{x_1x_2\cdots x_k\\ x_i\in P\setminus 1}}
\kern-2mm\cF(x_1),
\end{equation}
the direct sum over all multi-sequences $x_1\leq x_2\leq\cdots\leq x_k\in P\setminus 1$,
and $\cS_0(P,\cF) = \cF(1)$. The differential
$d_k\colon \cS_k(P,\cF) \ra \cS_{k-1}(P,\cF)$ is defined for $k>1$ by
\[
d_k(\lambda x_1x_2\cdots x_k)= \cF(x_1\leq x_2)(\lambda)x_2\cdots x_k 
- \sum_{i=2}^k(-1)^i\lambda x_1 \cdots \widehat{x}_i \cdots x_k,
\]
and $d_1$ is defined by $d_1(\lambda x) = \cF_{x}^{1}(\lambda).$
The homology $\hpc *$ of $(P,\cF)$ is then the homology of this complex,
and this gives a functor $H_*\colon \cpr \ra \grrmod$.

When working with homology it can be convenient to use a smaller complex than the one above. 
One defines $\cC_*(P,\cF)$
identically to $\cS_*(P,\cF)$, but with the additional requirement
that the 
indexing $x_1x_2\cdots x_k$ appearing above are now \emph{sequences\/} rather than multi-sequences, 
i.e. there are no repeats.
Clearly $\CC_k\subset\SS_k$, and the differential is taken to be the
restriction to $\CC_k$ of the differential on $\SS_k$ (so that
$\CC_*$ is a subcomplex of $\SS_*$). One can show that there is a homotopy equivalence of chain 
complexes $\cpc * \simeq \spf *$ as in \cite{Everitt-Turner}*{\S 2}.

Returning to the situation of 
a bundle $\xi:B\rightarrow\cpr$, these constructions give important colourings of the base.
Let $\EE(\xi)=(E,\cF)$. For each positive integer
$q\geq 0$ we have the \emph{$q$-chain colouring} $\SS_q$ of $B$: for
$x\in B$ let $\cS_q(x):=\cS_q(E_x,\FF_x)$, the module of $q$-chains in
the complex $\cS_*$ of the fibre $(E_x,\FF_x)$. Recalling from
\cite{Everitt-Turner}*{\S 2} that $\cS_*$ is a functor $\cS_*\colon
\cpr \ra \chr$, we have that if $x\leq y$ in $B$, then the composition of
$\xi\colon B \ra \cpr$ with $\cS_q$ induces a module homomorphism
$\cS_q(x)\rightarrow\cS_q(y)$, hence gives a colured poset $(B,\cS_q)$.

We can also colour the base using the \emph{homology of the fibres\/}.
Let $\xi:B\rightarrow\cpr$ with $\EE(\xi)=(E,\cF)$. 
For any $q$ compose $\xi\colon B \ra \cpr$ with homology 
$H_q\colon \cpr \ra \rmod$ to get a new functor $\cH^{\text{fib}}_q (\xi) = H_q \circ \xi$,
$$
\xymatrix{
\cH^{\text{fib}}_q (\xi)\colon B \ar[r]^-\xi & \cpr \ar[r]^-{H_q} & \rmod .
}
$$
This defines a colouring on $B$ and so a coloured poset $(B, \cH^{\text{fib}}_q(\xi))$.
If $(g,\eta):\xi\rightarrow\xi'$ is a bundle morphism and $x\in B$,
then $\eta(x):\xi(x)\rightarrow \xi'g(x)$ induces
a map in homology
$$
\eta(x)_*:H_q(\xi(x))\rightarrow H_q(\xi'g(x)),
$$
which in turn induces a 
morphism of coloured posets $(B,\cH_q^{\text{fib}}(\xi))\rightarrow (B',\cH_q^{\text{fib}}(\xi'))$.
Thus by taking homology of fibres we have a functor $ \bundr \ra \cpr$.

\section{A bicomplex and its total complex}
\label{section300}

Our main theorem is a Leray-Serre type spectral sequence for bundles.
Fundamental to its construction is
a certain bicomplex which we now describe. Let $\xi:B\rightarrow\cpr$ be a 
bundle with total coloured poset $\EE(\xi)=(E,\cF)$. From the previous section
we have an infinite family of coloured posets $(B,\cS_q)$, where the base 
has been $q$-chain coloured for each $q\geq 0$.

Now set 
\[
\cK_{p,q} = \cC_p(B, \cS_q).
\]

Explicitly, $\cC_p(B,\cS_q)$ is a direct sum of modules in the colouring of $B$ indexed by
length $p$ sequences in $B\setminus 1$, so a typical element is a sum
of $\mu x_1\ldots x_p$ with the 
$x_1<\cdots<x_p\in B\setminus 1$ and $\mu\in\cS_q(E_{x_1},\FF_{x_1})$.
Thus $\mu$ is in turn a sum over $\lambda y_1\ldots y_q$ with 
the $y_1\leq\cdots\leq y_q\in E_{x_1}\setminus 1_{x_1}$
and $\lambda\in\FF_{x_1}(y_1)$. We will thus write
$$
\lambda\x\y=\lambda x_1\ldots x_p y_1\ldots y_q,\hspace{1cm}\lambda\in\FF_{x_1}(y_1),
x_i\in B\setminus 1, y_j\in E_{x_1}\setminus 1_{x_1},
$$
for a typical (homogeneous) element of $\cC_p(B,\cS_q)$. 
Note that $\cK_{0,q}=\cC_0(B,\cS_q)=\cS_q(1)=\bigoplus\FF_1(y_1)$, the direct sum over
the multi-sequences $y_1\ldots y_q\in P_1$.

For fixed $q$ we have that 
$\cK_{*,q}= \cC_*(B,\cS_q)$ is a complex under the differential
$d^h:\cC_p(B,\cS_q)\rightarrow\cC_{p-1}(B,\cS_q)$ given by
\begin{align*}
d^h(\lambda x_1\ldots x_p y_1\ldots y_q)=&
(\tau_{x_1}^{x_2})_{y_1}(\lambda)x_2\ldots x_pf_{x_1}^{x_2}(y_1)\ldots f_{x_1}^{x_2}(y_q)\\
&-\sum_{i=2}^p(-1)^i\lambda x_1\ldots \widehat{x}_i\ldots x_p y_1\ldots y_q,
\end{align*}
where $(f_{x_1}^{x_2},\tau_{x_1}^{x_2}):(E_{x_1},\FF_{x_1})\rightarrow(E_{x_2},\FF_{x_2})$
is the induced coloured poset morphism. 

Fixing $p$ we have 
a complex under the differential 
$d^v:\cC_p(B,\cS_q)\rightarrow\cC_p(B,\cS_{q-1})$,
\begin{align*}
d^v(\lambda x_1\ldots x_p y_1\ldots y_q)=
(-1)^{p+q}&\biggl((\FF_{x_1})_{y_1}^{y_2}(\lambda)x_1\ldots x_py_2\ldots y_q\\
&-\sum_{j=2}^q(-1)^j\lambda x_1\ldots x_p y_1\ldots \widehat{y}_j\ldots y_q\biggr).
\end{align*}
It is tedious but straightforward to check that $d^vd^h+d^hd^v = 0$, so that,

\begin{lemma}
$\cK_{*,*}$ is a bi-complex.
\end{lemma}

There is then an associated total complex $(\TT_*, d)$ where
$$
\TT_n(E,\FF):=\bigoplus_{p+q=n} \cK_{p,q}
$$
with $d=d^h+d^v$. From the general construction of a spectral sequence from a bi-complex
we have,

\begin{proposition}\label{prop:ssT}
Let $\xi$ be a bundle of coloured posets with base $B$ and total poset $(E,\cF)$. 
Then there is a spectral sequence
\[
E^2_{p,q} = H_p(B, \cH^{\text{{\rm fib}}}_q (\xi)) \Longrightarrow H_*(\cT_*(E,\cF))
\]
\end{proposition}


\begin{example}\label{ex:productsacyclic}
Let $B$ be a poset with a unique minimal element 0 satisfying $0<1$,
and let $(P,\cF)$ be a coloured poset. Then $\cT_*(B\times (P,\cF))$
is acyclic. To see this notice that $(B, \cH^\text{fib}_q(\xi))$ is a coloured
poset with a constant colouring (by $H_q(P,\cF)$) and thus, since $B$
has a 0, gives \cite{Everitt-Turner}*{Example 8} that the $E_2$-page
of the above spectral sequence is trivial.
\end{example}

\section{Admissible and specially admissible posets}\label{section400}

In \S\ref{sec:main} we will identify the $H_*(\cT_*(E,\FF))$ 
of \S\ref{section300}
with the $H_*(E,\FF)$ of \S\ref{sec:bund} 
for $(E,\cF)=\EE(\xi)$, where
$\xi$ is a bundle over a certain class of posets that we now introduce.
Let $P$ be a poset with $1$, and for $x\in P$ recall from \S\ref{sec:bund} 
the definition of the interval
$P(x)$ and its complement $\ov{P}(x)$. 

\begin{definition}\label{definition400}
A poset $P$ with $1$ is \emph{admissible\/} if and only if there exists 
an $x\prec 1$ in $P$ such that for
all $y\in P(x)\setminus x$, the subposet
$L(y):=\{z\in\ov{P}(x)\,|\,y\leq z\}$
has a unique minimal element $0_{L(y)}$.
\end{definition}

\begin{figure}
  \centering
\begin{pspicture}(12.5,5.5)
\rput(-2,.5){
\rput(3,2){\BoxedEPSF{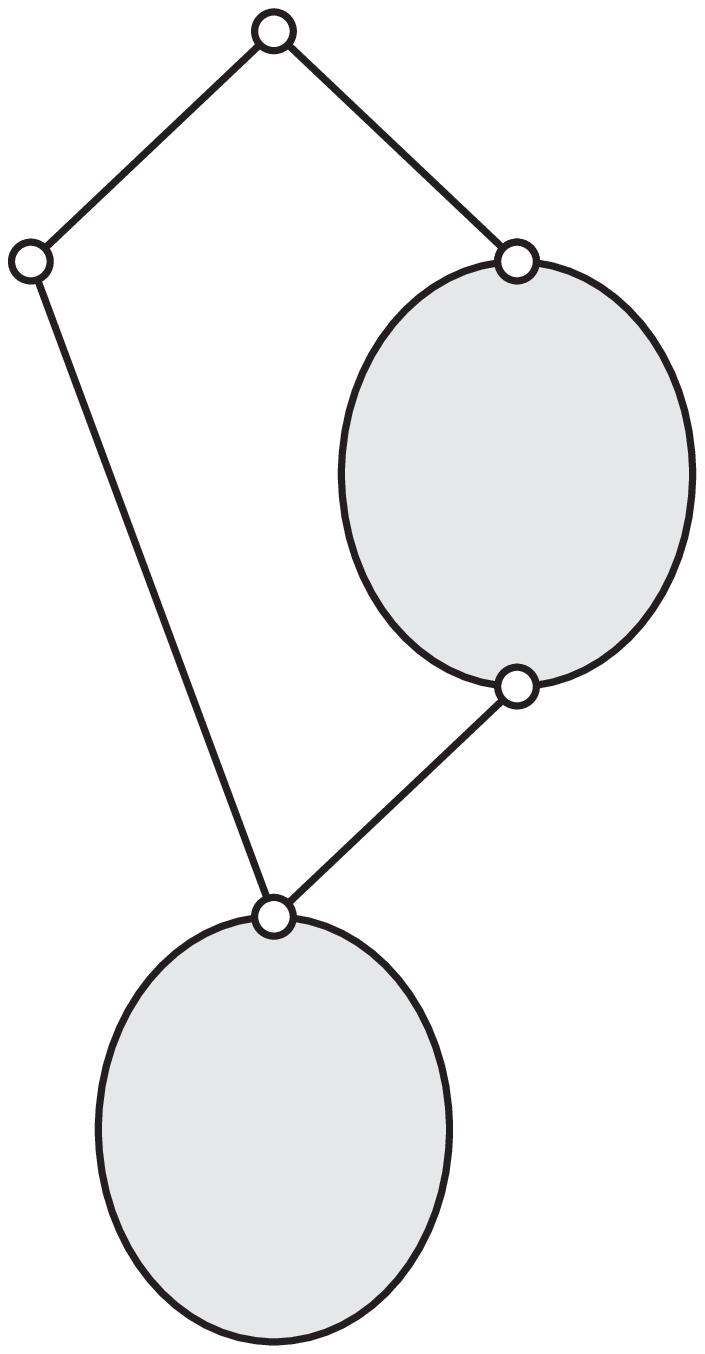 scaled 300}}
\rput(2.75,.6){$R$}\rput(3.5,2.6){$Q$}\rput(1.8,3.25){$x$}
}
\rput(4,2.5){\BoxedEPSF{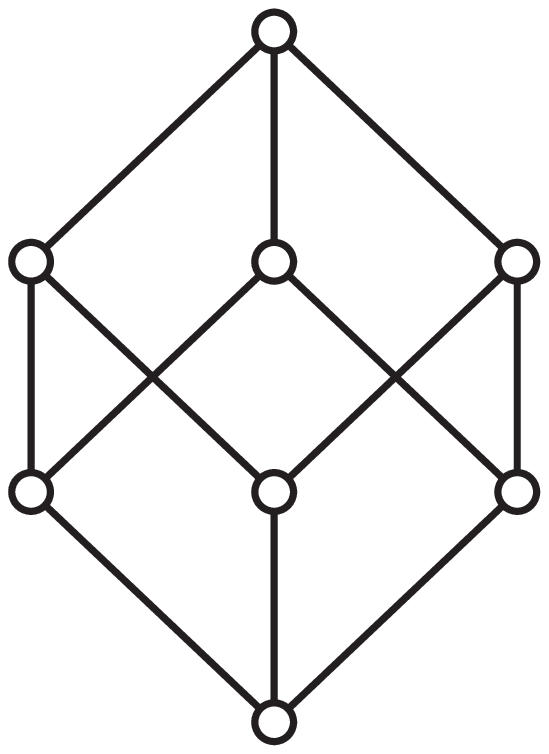 scaled 400}}
\rput(10,2.75){\BoxedEPSF{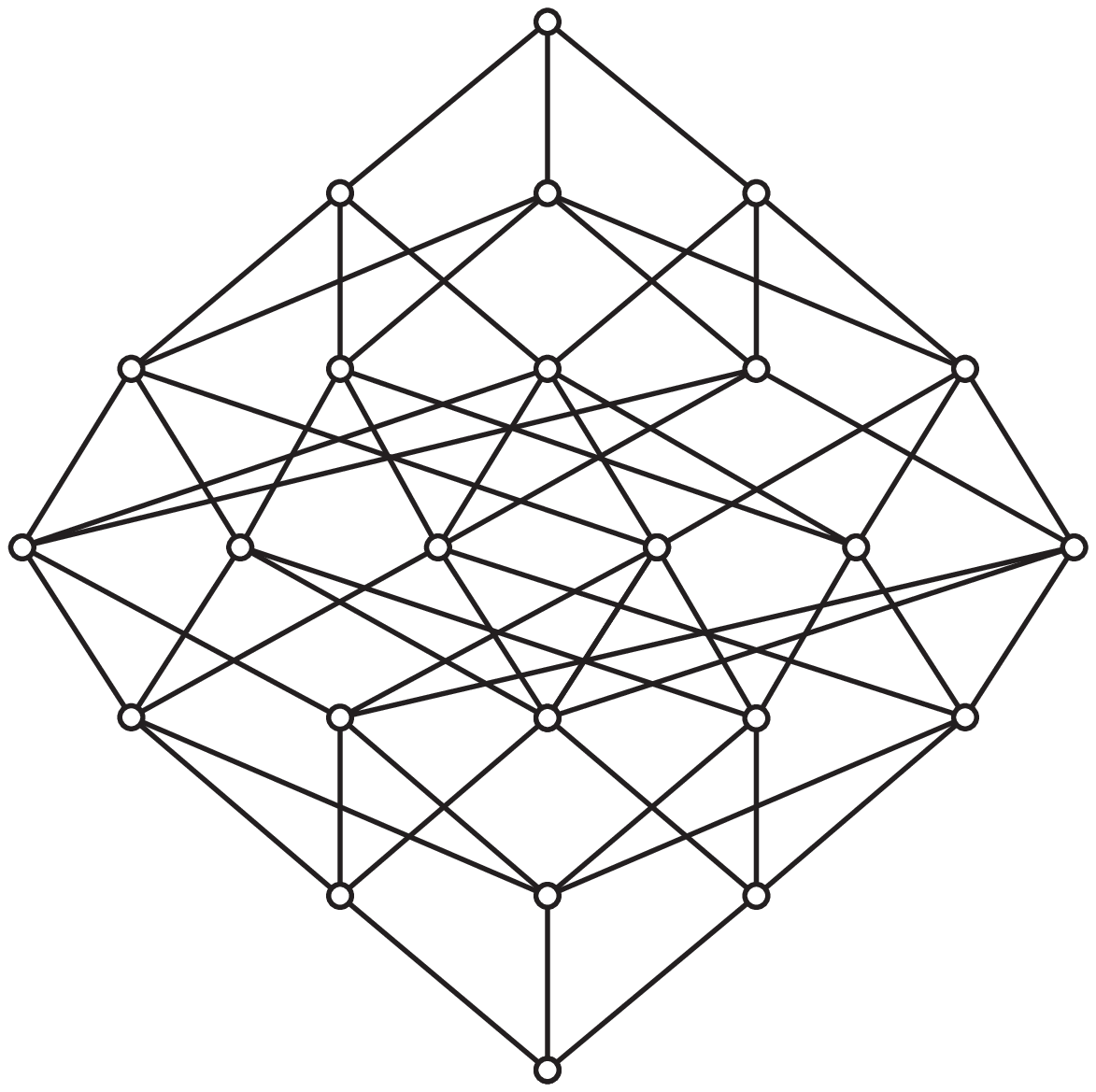 scaled 450}}
\rput(6.25,2.5){\BoxedEPSF{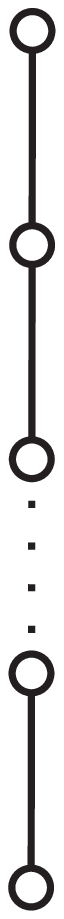 scaled 400}}
\end{pspicture}
  \caption{One of these posets is
non-admissible (from left to right): a concocted example where $R$ is any poset with
a $1$ and $Q$ is any poset with a $0<1$; 
the Boolean lattice of rank $3$;
the poset $\mathbf{n}=1<2<\cdots<n$;
the Bruhat poset of type $A_3$.}
  \label{figure400}
\end{figure}

On the far left of Figure \ref{figure400} we have an artificial
example of a poset, admissible courtesy of the $x$ shown. In general,
$L(y)$ has unique maximal element the $1$ of $P$ with
$0_{L(y)}<1$ by our standing poset convention, so $L(y)$ necessarily
contains at least two elements if $P$ is to be admissible. In
particular the poset $\mathbf{n}=1<2<\cdots<n$, with Hasse diagram
in Figure \ref{figure400} is non-admissible when
$n\geq 3$.

\begin{example}
For $\B(X)$ the Boolean lattice of all subsets of $X$ ordered by inclusion,
$|X|$ is the rank of $\B(X)$. Then $\B(X)$ is admissible for
$X\not=\varnothing$. This is vacuous in rank $1$ (where $\B(X)$ is isomorphic
to the poset $\mathbf{2}=1<2$). If $|X|>1$ then 
any subset $A\in\B(X)$ covered by $X$ is of the form $X\setminus x$ for some $x\in X$, and the
interval (also a Boolean lattice) $\B(A)$ consists of subsets of $X$
not containing $x$. For $Y\in \B(A) \setminus \{ A \}$ the set $L(Y)$
consists of subsets of $X$ containing the element $x$ and the elements
of $Y$. Thus $O_{L(Y)}= Y \cup \{ x \}$ is a unique minimal element
for $L(Y)$. The second example of Figure \ref{figure400} is Boolean of rank $3$.
\end{example}

\begin{example}
Let $(W,S)$ be a Coxeter group with length function
$\ell:W\rightarrow\Z^{\geq 0}$ and reflections $T=W^{-1}SW$ (see
\cite{Humphreys90} for basic facts about Coxeter
groups).  Write $w'\rightarrow w$ iff $w=w't$ for some $t\in T$ with
$\ell(w)>\ell(w')$.
Define $w'<w$ iff there is a sequence 
$w'=w_0\rightarrow w_1\rightarrow\cdots\rightarrow w_m=w$. The result is a graded 
(by $\ell$) poset with $0$ (the identity of $W$) called the \emph{Bruhat\/}
(-\emph{Chevelley}) poset of $W$. A simple example is the 
Bruhat poset of the Weyl group of type $A_1^n$, which is the Boolean lattice
of rank $n$. Two warnings: $w'\rightarrow w$ is \emph{not\/} the same as
$w'\prec w$, and $(W,\leq)$ is \emph{not\/} in general a lattice. 

If $W$ is finite then the Bruhat poset also has a $1$: the element of longest length
in $W$. 
It turns out that the Bruhat posets for finite $W$ are admissible, though we do not require that generality here.
For now, we leave it to the reader to check that
the Bruhat poset for the symmetric group $\gS_4$ (the Weyl group of type $A_3$),
which is the third example in Figure \ref{figure400}, is admissible for any
$x\prec 1$.
\end{example}

Suppose now that $\xi:B\rightarrow\cpr$ is a bundle with admissible base $B$ and
$(E,\FF)=\EE(\xi)$. The admissibility
of the base has the following useful consequence for $E$. 
Let $x\in B$, with interval $B(x)$, complement $\ov{B}(x)$, and restricted
bundles (see Example \ref{ex:int}) $\xi|_{B(x)}$, $\xi|_{\ov{B}(x)}$ with
$E(x)=\EE(\xi|_{B(x)})$ and $\ov{E}(x)=\EE(\xi|_{\ov{B}(x)})$.

\begin{lemma}\label{lemma400}
Let $B$ be an admissible poset courtesy of $x\prec 1_B$. Then for all
$y\in E(x)\setminus 1_{E(x)}$, the subposet $J(y)= \{z\in\ov{E}(x)\,|\,y\leq z\}$ has a 
unique minimal element.
\end{lemma}

\begin{proof}
We have $\pi(y)\in B(x)$ and
by admissibility, the subposet $L(\pi(y))$ has a unique minimal element $z_0\in\ov{B}(x)$.
The bundle furnishes us with a morphism of coloured posets 
$(E_{\pi(y)},\FF_{\pi(y)})\rightarrow (E_{z_0},\FF_{z_0})$, which is comprised in particular of 
a poset map $f_{\pi(y)}^{z_0}:E_{\pi(y)}\rightarrow E_{z_0}$. We leave it to the reader to
verify that the minimal element we seek is $f_{\pi(y)}^{z_0}(y)\in E_{z_0}$.
\end{proof}
 
Admissible posets are an important intermediary concept: the bundle results of 
\S\ref{sec:les} are for instance true when the base is admissible. For the main theorem
of \S\ref{sec:main} we require something stronger: that the base poset $B$ is admissible 
for some $x\prec 1$, and when 
split into the associated interval $B(x)$ and its complement $\ov{B}(x)$,
these two are also admissible,
and when these are split the results are admissible, and so on. Thus, we will
employ an inductive approach, resulting in a collection of posets scattered 
on the workbench, and we will need each piece to be admissible. The
definition is best formulated recursively:

\begin{definition}
A poset $P$ with $1$ is \emph{specially admissible\/} if and only if 
either
\begin{itemize}
\item $P$ is Boolean of rank $1$, or
\item $P$ is admissible for some $x\prec 1$ with $P(x)$ and $\ov{P}(x)$
specially admissible. 
\end{itemize}
\end{definition}

\begin{figure}
  \centering
\begin{pspicture}(12.5,2.5)
\rput(-1.25,0){
\rput(7.5,1.25){\BoxedEPSF{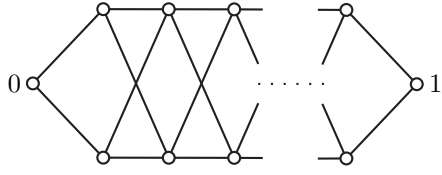 scaled 400}}
\rput(4.7,1.25){$0$}\rput(10.3,1.25){$1$}
}
\end{pspicture}
  \caption{Hasse diagram for the Bruhat poset of type $I_2(m)$, laid out on its
side with the ordering running from left to right.}
  \label{figure500}
\end{figure}

\begin{example}
We have already seen that the Boolean lattice of rank $n>0$ is
admissible (for any $x\prec 1$) with the $P(x)$ and
$\ov{P}(x)$ Boolean of rank $n-1$.  The Boolean
lattice of rank $1$ is specially admissible by definition, so Boolean
lattices of any rank are specially admissible by induction.
\end{example}

\begin{example}
Another simple family of specially admissible posets are the Bruhat posets
of type $I_2(m), m\geq 2$, which have Hasse diagrams (laid out sideways with order running
from left to right) as in Figure \ref{figure500}.
For any $x\prec 1$, the interval $P(x)$ is the Bruhat poset of type $I_2(m-1)$, and its 
complement $\ov{P}(x)$ is Boolean of rank $1$. As any two elements in $I_2(m)$ are
comparable, we have that for
$y\in P(x)\setminus x$, the poset $L(y)$ is also Boolean of rank $1$. Thus
$I_2(m)$ is admissible. The Bruhat poset of type $I_2(2)$
is the Boolean lattice of rank $2$, and so by induction, $I_2(m)$ is specially 
admissible.
\end{example}

\begin{example}
Looking back at Figure \ref{figure400}, one can see (after
quite an inspection) that the Bruhat poset $A_3$ is 
specially admissible. 
Indeed, 
it turns out that the Bruhat posets for finite $W$ are
specially admissible.
\end{example}

\section{Some exact sequences}\label{sec:les}

This section contains the technical tools needed for the proof of
the main theorem. As we will see in \S\ref{sec:main}, the key step is
the identification of the
homology of $\cT_*(E,\cF) $ with the
coloured poset homology $H_*(E,\cF)$ in the case where the 
base poset is specially admissible. To achieve this we will use an inductive
argument, splitting the base into two pieces and then using the two long
exact sequences presented below in Propositions \ref{prop:lest} and \ref{prop:less}.

Throughout this section let $B$ be a poset, admissible courtesy of some
$x\prec 1_B$. 
Let $\xi$ be a bundle of coloured posets with base $B$ and $\EE(\xi)=(E,\cF)$.
Let $B(x)$ be the interval from \S\ref{sec:bund} and 
$\ov{B}(x)$ its complement, with the associated restricted bundles having total posets
$(E(x), \cF)$ and $(\ov{E}(x), \cF)$. 

\begin{proposition}\label{prop:lest}
Let $\xi:B\rightarrow\cpr$ be a bundle of coloured posets with 
$B$ admissible, $(E,\cF)$ the total coloured poset and $\cT_*(E,\cF)$ the total complex of
\S$\ref{section300}$.
Then there is a long exact sequence,
$$
\xymatrix{
\cdots\ar[r]^-{\delta}
& { H_n(\cT_*(\oE(x), \cF))} \ar[r]
& { H_n (\cT_*(E, \cF))} \ar[r]
& { H_{n-1}(\cT_*(E(x), \cF))} \ar[r]^-{\delta}
& \cdots}
$$
\end{proposition}

\begin{proof}
It is immediate from the definition of the differential that $ \cT_*(\oE(x), \cF)$ is a 
subcomplex of  $ \cT_*(E, \cF)$ and thus for some quotient 
$Q^\prime_*$ there is a short exact sequence
\begin{equation}\label{eq:sest}
\xymatrix{
0  \ar[r]
& \cT_*(\oE(x), \cF) \ar[r]
& \cT_*(E, \cF) \ar[r]
& Q^\prime_* \ar[r]
& 0.
}
\end{equation}
Explicitly, $Q^\prime_* $ can be described as
$$
Q^\prime_n = \bigoplus_{\bx\bz\by} \cF_{x_1}(y_1),
$$ 
the sum over $\bx\bz\by = x_1\ldots x_iz_1 \ldots z_jy_1\ldots y_k$ with $i>0$
and $i+j+k=n$, and where $x_1\ldots x_i$ is a sequence in $B(x)$, $z_1
\ldots z_j$ a sequence in $\ov{B}(x)\setminus 1$, and $y_1\ldots y_k$ a
multi-sequence in $E_{x_1}\setminus 1_{E_{x_1}}$. Here, as elsewhere $E_{x_1}$
is the fibre over $x_1$. Notice that the
condition $i>0$ means that the sequence $ x_1\ldots x_iz_1 \ldots z_j$
must start in $B(x)$.  
When $i>1$ the differential in $Q'_*$ sends $\lambda\x\y\z$ to,
\begin{align*}
(\tau_{x_1}^{x_2}&)_{y_1}(\lambda) x_2\ldots x_iz_1\ldots z_jf_{x_1}^{x_2}(y_1)
\ldots f_{x_1}^{x_2}(y_k)\\
&-\sum_{s=2}^i(-1)^s\lambda x_1\ldots\widehat{x}_s\ldots x_iz_1\ldots z_jy_1\ldots y_k\\
&\quad -\sum_{t=1}^j(-1)^{i+t}\lambda x_1\ldots x_iz_1\ldots\widehat{z}_t\ldots z_jy_1\ldots 
y_k\\
& \quad\quad +(-1)^n(\cF_{x_1})_{y_1}^{y_2}(\lambda)  x_1\ldots x_iz_1 \ldots z_jy_2\ldots y_k\\
&\quad\quad\quad -\sum_{\ell=2}^k(-1)^{\ell+n}\lambda x_1\ldots x_iz_1\ldots z_jy_1
\ldots\widehat{y}_\ell\ldots y_k,
\end{align*}
and for $\lambda xz_1\ldots z_jy_1\ldots y_k$ the image is
\begin{align*}
\sum_{t=1}^j (-1&)^{t} \lambda  xz_1 \ldots \widehat{z}_t \ldots z_jy_1\ldots y_k\\
& + (-1)^{n}(\cF_{x})_{y_1}^{y_2}(\lambda)  xz_1 \ldots z_jy_2\ldots y_k\\
& \quad -\sum_{\ell=2}^k (-1)^{\ell+n} \lambda  xz_1 \ldots z_jy_1\ldots  \widehat{y}_\ell \ldots y_k.
\end{align*}

Now let $D^\prime_*$ be the submodule of $Q^\prime_*$ 
consisting of the modules indexed by the $\bx\bz\by
= x_1\ldots x_iz_1$ $\ldots z_jy_1\ldots y_k$ with $x_i\neq x$. In
other words, the last element of the sequence in $B(x)$ is not the
maximum element $x$ of this interval. 
It is immediate that  $D^\prime_*$ is a subcomplex of $Q'_*$ and thus for the
quotient $A^\prime_*$ there is a short exact sequence
\begin{equation}\label{eq:sesdqa}
\xymatrix{
0  \ar[r]
& D^\prime_* \ar[r]
& Q^\prime_* \ar[r]^-{q'}
& A^\prime_* \ar[r]
& 0.
}
\end{equation}

We show that $D^\prime_*$ is acyclic, and this allows us to
finish the proof easily: one sees that 
$$
\pi':A^\prime_{i+k} \rightarrow
\cT_{i+k-1}(E(x), \cF)
$$ 
given by 
$\lambda x_1\ldots x_{i-1}xy_1\ldots y_k\mapsto\lambda x_1\ldots x_{i-1}y_1\ldots y_k$
is an isomorphism,
and so using the long exact sequence associated
to (\ref{eq:sesdqa}), we have $H_n(Q^\prime_*) \cong H_{n-1}(
\cT_{*}(E(x), \cF))$. Plugging this into the long exact sequence associated 
to (\ref{eq:sest}) gives the result.

The remainder of the proof is thus devoted to showing that $D^\prime_*$ is acyclic. 
We filter  $D^\prime_*$ as follows:
$$
F_sD^\prime_n = \bigoplus_{\bx\bz\by} \cF_{x_1}(y_1)
$$ 
where the sum is over the $\bx\bz\by = x_1\ldots x_iz_1 \ldots z_jy_1\ldots
y_k$ satisfying $1\leq i \leq s$. 
There is an associated spectral sequence converging to $H_*(D^\prime_*)$ having
$$
E^0_{s,t} =  \bigoplus_{\bx\bz\by} \cF_{x_1}(y_1),
$$ 
where the sum is over the $\bx\bz\by = x_1\ldots x_sz_1 \ldots z_jy_1\ldots
y_k$ with $j+k=t$. The differential $d^0$ sends 
$\lambda x_1\ldots x_sz_1 \ldots z_jy_1\ldots y_k$ to
\begin{align*}
-\sum_{i=1}^j (-1&)^{s+i} \lambda  x_1\ldots x_sz_1 \ldots \widehat{z}_i \ldots z_jy_1\ldots y_k\\
& + (-1)^n(\cF_{x_1})_{y_1}^{y_2}(\lambda)  x_1\ldots x_sz_1 \ldots z_jy_2\ldots y_k\\
&\quad -\sum_{i=2}^k (-1)^{n+i} \lambda  x_1\ldots x_sz_1 \ldots z_jy_1\ldots  \widehat{y}_i \ldots y_k.
\end{align*}
Given $\bx = x_1 \ldots x_s$ a sequence in $B(x) \setminus x$ we define 
$U^{\x}_* \subseteq E^0_{s,*}$ to be the submodule of elements of the 
form $\lambda x_1\ldots
x_sz_1 \ldots z_jy_1\ldots y_k$, where the sequence $x_1\ldots x_s$ is the
given one $\x$. Endowing this with the differential gives
a subcomplex.

Let $L(x_s)$ be the subposet of $B$ with elements
$\{ z\in \ov{B}(x) \mid x_s\leq z \}$ as in Definition \ref{definition400},
and having a unique minimal element courtesy of 
the admissibility of $B$.
By direct comparison of the definitions on both sides we see that there is an isomorphism 
of complexes 
$$
U^{\bx}_* \cong \cT_{*-s}(L(x_s) \times E_{x_1}),
$$
and by Example \ref{ex:productsacyclic}, the right hand side (and hence $U^{\bx}_*$),
is acyclic.

We now show that any cycle in $E^0_{s,*}$ is a boundary. 
A typical cycle has the form $\sigma = \sum \lambda_{\bx\bz\by} \bx \bz \by$ 
and can be decomposed as
\[
\sigma = \sum \sigma^{\bx},
\]
where $\sigma^{\bx} = \sum_{\bx^\prime = \bx}
\lambda_{\bx^\prime\bz\by} \bx^\prime \bz \by$. Now, 
if $\bx\bz\by$ and
$\bx^\prime\bz^\prime\by^\prime$ are two sequences occuring in
$\sigma$ and $\bx\neq \bx^\prime$ then $d^0(\bx\bz\by)$ and
$d^0(\bx^\prime\bz^\prime\by^\prime)$ are in disjoint summands of
$d(\sigma)$. Thus $d^0(\sigma)=0$ implies that $d^0(\sigma^\bx)=0$ for
all $\sigma^\bx$. In other words, as $\ss$ is a cycle, each $\sigma^\bx$ is a cycle. 
Clearly we have $\sigma^\bx \in U^\bx_*$ and hence since $U^\bx_*$ is
acyclic, there exists $\tau^\bx$ such that
$d^0(\tau^\bx) = \sigma^\bx$. Thus $d^0(\sum \tau^\bx) = \sum
\sigma^\bx = \sigma$ and $\sigma$ is a boundary as claimed.

Thus, for each $s$ the complex $E^0_{s,*}$ is acyclic and hence
the $E^1$-page of the spectral sequence is trivial, showing that
$D^\prime_*$ is acyclic as required.
\end{proof}

In the proof of the main theorem in \S\ref{sec:main}, we will find it more convenient
to use 
$$
(-1)^k\pi':\lambda x_1\ldots x_{i-1}xy_1\ldots y_k\mapsto(-1)^k\lambda x_1\ldots x_{i-1}y_1\ldots y_k
$$ 
as an isomorphism $A_*'\rightarrow\cT_*(E(x),\cF)$.
We leave it to the reader to check that this is indeed an isomorphism.

There is a similar long exact sequence for the homology of the total coloured
poset $(E,\FF)$. It is a generalization of \cite{Everitt-Turner}*{Theorem 1}
from bases that are Boolean of rank $1$ to bundles over admissible bases.
The proof is very similar to the last proposition so we will be briefer.

\begin{proposition}\label{prop:less}
Let $\xi:B\rightarrow\cpr$ be a bundle of coloured posets with $B$ admissible and
total coloured poset $(E,\cF)$. 
Then there is a long exact sequence,
$$
\xymatrix{
\cdots \ar[r]^-{\delta}
& {H_n(\oE(x), \cF)} \ar[r]
& {H_n (E, \cF)} \ar[r]
& {H_{n-1}(E(x), \cF)} \ar[r]^-{\delta}
& \cdots}
$$
\end{proposition}

\begin{proof}
 There is a short exact sequence
\begin{equation}\label{eq:sess}
\xymatrix{
0  \ar[r]
& \cS_*(\oE(x), \cF) \ar[r]
& \cS_*(E, \cF) \ar[r]
& Q_* \ar[r]
& 0,
}
\end{equation}
and as above we can describe $Q_* $ explicitly as 
$$
Q_n = \bigoplus_{\bx\bz} \cF(x_1),
$$ 
where $\bx\bz = x_1\ldots x_iz_1 \ldots z_j$ with $i>0$,
$i+j=n$ and $ x_1\ldots x_i$ a multi-sequence in $E(x)$ and $z_1
\ldots z_j$ a multi-sequence in $\oE(x)\setminus 1_E$. 

Now let $D_*$ be the subcomplex of $Q_*$ 
consisting of the modules indexed by the $\bx\bz
= x_1\ldots x_iz_1$ $\ldots z_j$ such that $x_i\neq 1_{E(x)}$,
and $A_*$ the quotient of $Q_*$ by $D_*$ with quotient map $q$. 
We show, as above, that $D_*$ is acyclic, so that if 
$$
\pi:A_{i+j}\rightarrow\cS_{i+j-1}(E(x),\cF)
$$
is the isomorphism $\lambda x_1\ldots x_{i-1}1_{E(x)}z_1\ldots z_j
\mapsto \lambda x_1\ldots x_{i-1}z_1\ldots z_j$, then 
$\pi q$ is a quasi-isomorphism from $Q_*$to $\cS_{*-1}(E(x),\cF)$.

Filter $D_*$ by
$$
F_sD_n = \bigoplus_{\bx\bz} \cF (x_1)
$$ 
where the sum is over the $\bx\bz = x_1\ldots x_iz_1 \ldots z_j$ 
satisfying $1\leq i \leq s$. 
The associated spectral sequence has
$$
E^0_{s,t} =  \bigoplus_{\bx\bz} \cF (x_1)
$$ 
where the sum is over $\bx\bz  = x_1\ldots x_sz_1 \ldots z_t$. 
The differential $d^0$ is as follows (where the
overall sign is given by the parity of $s$):
\begin{eqnarray*}
\pm d^0 ( \lambda x_1\ldots x_sz_1 \ldots z_t) & = & 
\sum_{k=1}^t (-1)^i \lambda  x_1\ldots x_sz_1 \ldots \widehat{z}_i \ldots z_t.
\end{eqnarray*}
Fixing $\bx = x_1 \ldots x_s$ a multi-sequence in $E_x \setminus 1_{E(x)}$, we
define $V^{\bx}_* \subseteq E^0_{s,*}$ to be the subcomplex of
elements of the form $\lambda x_1\ldots x_sz_1 \ldots z_j$ where the
multi-sequence $x_1\ldots x_s$ is the given one $\x$. Let
$J(x_s)$ be the subposet of $E$ with elements $ \{ z\in \oE_x \mid x_s
\leq z\}$.  By Lemma \ref{lemma400}, $J(x_s)$ has a 0 and 
$V^{\bx}_* \cong \cS_{*-s}(J(x_s), \cF(x_1))$. Since the poset $J(x_s)$
is coloured constantly by $\cF(x_1)$, we conclude by 
\cite{Everitt-Turner}*{Example 8 and Corollary 1} that
$V^{\bx}$ is acyclic.
The remainder of the proof is more or less identical to the previous
proposition.
\end{proof}

\section{A Leray-Serre spectral sequence}\label{sec:main}

Here is the main theorem of the paper.

\begin{theorem}\label{thm:main}
Let $\xi:B\rightarrow\cpr$ be a bundle of coloured posets with $B$ finite and
specially admissible and 
$(E,\cF)$ the associated total coloured poset. Then there is a spectral sequence
\[
E^2_{p,q} = H_p(B, \cH^{\text{{\rm fib}}}_q (\xi)) \Longrightarrow H_*(E,\cF).
\]
\end{theorem}

The strategy of the proof is as follows: 
in \S\ref{section300} we constructed a spectral sequence, which
by Proposition \ref{prop:ssT} converges,
$$
E^2_{p,q} = H_p(B, \cH^{\text{fib}}_q (\xi)) \Longrightarrow H_*(\cT_*(E,\cF)).
$$
It suffices then to find a quasi-isomorphism $\phi:\cT_*(E,\cF)\rightarrow\cS_*(E,\cF)$,
inducing an isomorphism $H_*(\cT_*(E,\cF))\cong H_*(\cS_*(E,\cF))=H_*(E,\FF)$.

We begin by defining a chain map $\phi:\cT_*(E,\cF)\rightarrow\cS_*(E,\cF)$.
Let $\xi:B\rightarrow\cpr$ be a bundle, with $B$ any base, and $(E,\FF)$ the total
coloured poset.
Let $(\x,\y)$ be a pair consisting of 
a sequence $\x=x_1<\cdots<x_n$ in
$B$ and a multi-sequence $\y=y_1\leq\cdots\leq y_m$ in $E_{x_1}$. Note that
we allow $1\in B$ to be an element of $\x$ and $1_{x_1}\in E_{x_1}$ to be an element of $\y$.
For each $1\leq i\leq n$, the bundle furnishes us with a poset map 
$f_{x_1}^{x_i}:E_{x_1}\rightarrow E_{x_i}$, so
let $y_{ij}=f_{x_1}^{x_i}(y_j)$, giving
\begin{equation}
  \label{eq:1}
y_{i1}\leq\cdots\leq y_{im},  
\end{equation}
a multi-sequence in $E_{x_i}$ (with $y_{1j}=y_j$). The resulting
$m\times n$ collection $\{y_{ij}\}$ may be placed on the vertices of a
rectangular array and we join these by placing oriented edges between
$y_{i,j}$ and $y_{i,j+1}$ and between $y_{i,j}$ to $y_{i+1,j}$. 
We refer to this as the \emph{grid\/}
associated to the pair $(\x,\y)$. 
We have comparability of elements in
the columns of the grid, $y_{ij}\leq y_{i,j+1}$, thanks to
(\ref{eq:1}), and also along the rows, $y_{ij}<y_{i+1,j}$, via the
definition of the ordering on $E$.

Given  $x_1<\cdots<x_p\in B\setminus 1$  and 
$y_1\leq\cdots\leq y_q\in E_{x_1}\setminus 1_{x_1}$
with $p,q>0$, we consider the grid associated to  
$(\x=x_1<\cdots<x_p<1,\y=y_1\leq\cdots\leq y_q<1_{x_1})$. 
An $(\x,\y)$-{\em multi-sequence} is a multi-sequence in 
$E$ of the form,
$$
\z=z_1\leq\cdots\leq z_{p+q},
$$
with the $z_k$ elements of the grid satisfying $z_1=y_{11}$, $z_{p+q}=y_{p+1,q}$ or 
$z_{p+q}=y_{p,q+1}$, and for any $1\leq k\leq p+q$, if
$z_k=y_{ij}$, then $z_{k+1}\in\{y_{i+1,j},y_{i,j+1}\}$. In other words an 
$(\x,\y)$-multisequence consists of the ordered collection of vertices in an 
oriented path from $y_{11}$ to either $y_{p+1,q}$ or $y_{p,q+1}$. 
Figure \ref{figure505} exhibits a grid and two sample paths.

\begin{figure}
  \centering
\begin{pspicture}(12.5,5)
\rput(-1.25,0){
\rput(4.875,.75){
\psline[linewidth=.3mm]{-}(0,0)(5.25,0)
\psline[linewidth=.3mm]{-}(0,.75)(5.25,.75)
\psline[linewidth=.3mm]{-}(0,1.5)(5.25,1.5)
\psline[linewidth=.3mm]{-}(0,2.25)(5.25,2.25)
\psline[linewidth=.3mm]{-}(0,3)(5.25,3)
\psline[linewidth=.3mm]{-}(0,3.75)(5.25,3.75)
\psline[linewidth=.3mm]{-}(0,0)(0,3.75)
\psline[linewidth=.3mm]{-}(.75,0)(.75,3.75)
\psline[linewidth=.3mm]{-}(1.5,0)(1.5,3.75)
\psline[linewidth=.3mm]{-}(2.25,0)(2.25,3.75)
\psline[linewidth=.3mm]{-}(3,0)(3,3.75)
\psline[linewidth=.3mm]{-}(3.75,0)(3.75,3,75)
\psline[linewidth=.3mm]{-}(4.5,0)(4.5,3.75)
\psline[linewidth=.3mm]{-}(5.25,0)(5.25,3.75)
\psline[linewidth=1mm,linestyle=dashed]{-}(0,0)(0,.75)(.75,.75)(.75,1.5)(2.25,1.5)(2.25,2.25)
(4.5,2.25)(4.5,3)(5.25,3)
\psline[linewidth=1mm]{-}(0,0)(1.5,0)(1.5,.75)(3,.75)(3,1.5)(3.75,1.5)
(3.75,3.75)(4.5,3.75)
}
\rput(4.9,.5){$x_1$}\rput(5.6,.5){$x_2$}\rput(6.4,.5){$x_3$}\rput(9.4,.5){$x_p$}
\rput(10.1,.5){$1$}
\rput(4.6,.8){$y_1$}\rput(4.6,1.5){$y_2$}\rput(4.6,3.75){$y_q$}\rput(4.6,4.5){$1_{x_1}$}
\rput(9.4,4.7){$y_{p,q+1}$}\rput(10.65,3.75){$y_{p+1,q}$}
\psline[linewidth=.5mm,linestyle=dotted](6.8,.5)(9,.5)
\psline[linewidth=.5mm,linestyle=dotted](4.6,1.8)(4.6,3.4)
}
\end{pspicture}
  \caption{$(\x,\y)$-grid with two posible multi-sequences $\z$.}
  \label{figure505}
\end{figure}

Such a path partitions the grid into two halves--the upper-left
half and the lower-right half. Given an $(\x,\y)$-multisequence $\z$ we define
$m(\z)$ to be the number of squares of the grid in the lower-right half.

We now define the map $\phi$ on the generator 
$\lambda \x \y = \lambda x_1\cdots x_py_1\cdots y_q$ by,
\begin{equation}
  \label{eq:2}
\phi(\lambda \x\y) = (-1)^{\aa(q)}\sum_{\z} (-1)^{m(\z)}\lambda \z,
\end{equation}
where the sum is over all $(\x,\y)$-multisequences $\z$, and 
$\aa(q)=1$ when $q\equiv 1,2\text{ mod }4$, or $\aa(q)=0$ otherwise 
(alternatively, $(-1)^{\aa(q)}=(-1)^{[q]}$ for $[q]=1+2+\cdots+q$).

Note that (\ref{eq:2}) is also defined when $p$ or $q$ are $0$: we have 
$\cT_{p+0}(E,\FF)=\cC_p(B,\cS_0)$ with a typical element $\lambda x_1\ldots x_p$
where $\lambda\in\FF_{x_1}(1_{x_1})$ and $\phi(\lambda\x)=\lambda 1_{x_1}\ldots 1_{x_p}$. 
Similarly
$\cT_{0+q}=\cC_0(B,\cS_q)$ with typical element $\lambda y_1\ldots y_q$ for 
$\lambda\in\FF_1(y_1)$, and $\phi(\lambda\y)=(-1)^{\aa(q)}\lambda y_1\ldots y_q$. 
Finally, note that $\cT_{0+0}=\FF_1(1_{P_1})=\cS_{0+0}$ so by convention we define
$\phi=\id:\cT_{0+0}\rightarrow\cS_{0+0}$.

\begin{proposition}\label{sec:main:result200}
$\phi$ is a chain map 
with
$\phi(\cT_*(\ov{E}(x),\FF))\subset\cS_*(\ov{E}(x),\FF)$ for $x\in B$. 
\end{proposition}

\begin{proof}
The inclusion of complexes is clear, so that it remains to show that
$\phi d_\cT(\lambda\x\y)$ $=d_\cS\phi(\lambda\x\y)$ for any $\lambda\x\y\in\cT_{p+q}(E,\FF)$.
There are two parts to the proof:
first we show that every term of $\phi d_\cT$ also occurs as a term of $d_\cS \phi$, and then
that any extra terms in $d_\cS\phi$ occur in $\pm$ pairs, and hence cancel. 

Ignoring signs for the
moment, the terms of $\phi d_\cT$ are precisely the terms 
of $d_\cS\phi$ that
are obtained by taking a 
$\z=z_1\ldots z_{p+q}$
of $\phi(\lambda\x\y)$ and letting the differential $d_\cS$
omit $z_i$ as,
$$
\begin{pspicture}(12.5,2)
\rput(3,0){
\psline[linewidth=.5mm]{->}(-.4,1.4)(2.4,1.4)
\psline[linewidth=.3mm]{-}(0,0)(0,1)(0,2)
\psline[linewidth=.3mm]{-}(1,0)(1,1)(1,2)
\psline[linewidth=.3mm]{-}(2,0)(2,1)(2,2)
\psline[linewidth=.3mm]{-}(0,0)(1,0)(2,0)
\psline[linewidth=.3mm]{-}(0,1)(1,1)(2,1)
\psline[linewidth=.3mm]{-}(0,2)(1,2)(2,2)
\rput*(0,1){$z_{i-1}$}\rput*(1,1.05){$\hat{z}_i$}\rput*(2,1){$z_{i+1}$}
}
\rput(7,0){
\psline[linewidth=.5mm]{->}(.5,-.4)(.5,2.4)
\psline[linewidth=.3mm]{-}(0,0)(0,1)(0,2)
\psline[linewidth=.3mm]{-}(1,0)(1,1)(1,2)
\psline[linewidth=.3mm]{-}(2,0)(2,1)(2,2)
\psline[linewidth=.3mm]{-}(0,0)(1,0)(2,0)
\psline[linewidth=.3mm]{-}(0,1)(1,1)(2,1)
\psline[linewidth=.3mm]{-}(0,2)(1,2)(2,2)
\rput*(1,0){$z_{i-1}$}\rput*(1,1){$\hat{z}_i$}\rput*(1,2){$z_{i+1}$}
}
\end{pspicture}
$$ 
or by letting $d_\cS$ omit the first element $z_1$ or the last
element $z_{p+q}$. In words, if $d_\cS$ omits $z_i$, then
$z_{i-1},z_{i+1}$ are both in the same row, or both in the same
column, as $z_i$. Indeed, the situation above left arises in $\phi
d_\cT$ when $d_\cT$ omits some $x\in B$ that indexes the column in
which $z_i$ lies, and then mapping to the multi-sequence
$\zhat=z_1\ldots \widehat{z}_i \ldots z_{p+q}$ corresponding to a path in the
grid in which the $i$-th column of the original grid has been
deleted. The situation above right is similar, with an appropriate
$y\in E_{x_1}$ omitted by $d_\cT$ this time. Chasing definitions one
can verify that the coefficients of such terms agree up to sign.

With that out of the way, the terms of $d_\cS\phi$ that don't arise at 
all in $\phi d_\cT$
are those where $d_\cS$ omits from a 
multi-sequence $\z$ a $z_i$ sitting on a ``corner'':
$$
\begin{pspicture}(12.5,2)
\rput(3.5,0.5){
\psline[linewidth=.3mm]{-}(0,0)(0,1)\psline[linewidth=.3mm]{-}(1,0)(1,1)
\psline[linewidth=.3mm]{-}(0,0)(1,0)\psline[linewidth=.3mm]{-}(0,1)(1,1)
\rput*(0,0){$z_{i-1}$}\rput*(0,1.05){$\hat{z}_i$}\rput*(1,1){$z_{i+1}$}
\psline[linewidth=.5mm]{->}(-.5,-.4)(-.5,1.4)(1.4,1.4)
}
\rput(7.5,0.5){
\psline[linewidth=.3mm]{-}(0,0)(0,1)\psline[linewidth=.3mm]{-}(1,0)(1,1)
\psline[linewidth=.3mm]{-}(0,0)(1,0)\psline[linewidth=.3mm]{-}(0,1)(1,1)
\rput*(0,0){$z_{i-1}$}\rput*(1,0.05){$\hat{z}_i$}\rput*(1,1){$z_{i+1}$}
\psline[linewidth=.5mm]{->}(-.4,-.4)(1.5,-.4)(1.5,1.4)
}
\end{pspicture}
$$
Now to the signage, where we check first that 
these extra terms occur in $\pm$
pairs and so cancel. Suppose we have such a term, arising when
$d_\cS$ omits a corner from a $\z\in\phi(\lambda\x\y)$. There is
a unique repetition of this term, arising 
by taking the $\z'\in\phi(\lambda\x\y)$ that travels the other way around this square
and letting $d_\cS$ omit the other corner ($\z$ and $\z'$ thus differ only
in the segment shown in the figure above).
It is easy to check that $\z,\z'$ acquire different signs from $\phi$ and that
$d_\cS$ preserves this difference. Thus, the extra terms in $d_\cS\phi$ can be
coupled into $\pm$ pairs as required.

Now to the terms common to $\phi d_\cT$  
and $d_\cS\phi$, where we show that they
occur with the same signs. Let $\lambda \x\y =  
\lambda x_1\cdots x_py_1 \cdots  y_q$ and apply $\phi$ to 
give a sum of $(\x,\y)$-multi-sequences with the multi-sequence 
$\z$ having sign $(-1)^{\aa(q)+ m(\z)}$. Now applying $d_\cS$ to this 
we get a sum of terms of the form $z_1\ldots \hat{z}_i \ldots z_{p+q}$ picking 
up a sign $-(-1)^i$. Thus the total sign of the summand of $d_\cS\phi (\lambda\x\y)$ 
indexed by  $z_1\ldots \hat{z}_i \ldots z_{p+q}$ is
\begin{equation}\label{eq:sign1}
-(-1)^{\aa(q)+ m(\z)+i}.
\end{equation}
Now there are two cases to consider, where the first is if
$z_{i-1},z_{i}$ and  $z_{i+1}$ are all in the same column (see Figure \ref{figure510}). 

\begin{figure}
  \centering
\begin{pspicture}(12.5,5)
\rput(-1.75,0){
\rput(4.875,1.25){
\psline[linewidth=.3mm]{-}(0,-.75)(6,-.75)
\psline[linewidth=.5mm,linestyle=dotted]{-}(0,.75)(6,.75)
\psline[linewidth=.5mm,linestyle=dotted]{-}(0,1.5)(6,1.5)
\psline[linewidth=.5mm,linestyle=dotted]{-}(0,2.25)(6,2.25)
\psline[linewidth=.3mm]{-}(0,3.75)(6,3.75)
\psline[linewidth=.3mm]{-}(0,-.75)(0,3.75)
\psline[linewidth=.5mm,linestyle=dotted]{-}(3,-.75)(3,3.75)
\psline[linewidth=.3mm]{-}(6,-.75)(6,3.75)
\psline[linewidth=.7mm]{->}(2.25,0)(3,.75)(3,2.25)(3.75,3)
}
\rput(7.9,.2){$x_k$}\rput(4.35,2.75){$y_{i-k+1}$}
\rput(7.55,2.15){$z_{i-1}$}\rput(7.65,2.9){$z_i$}\rput(7.55,3.65){$z_{i+1}$}
\rput(9.4,1.25){$(p+1-k)$}
\psline[linewidth=.3mm]{->}(8.5,1.25)(7.9,1.25)
\psline[linewidth=.3mm]{->}(10.3,1.25)(10.9,1.25)
}
\end{pspicture}
  \caption{schematic $\z$ with $z_{i-1},z_i,z_{i+1}$ all in the same column.}
  \label{figure510}
\end{figure}
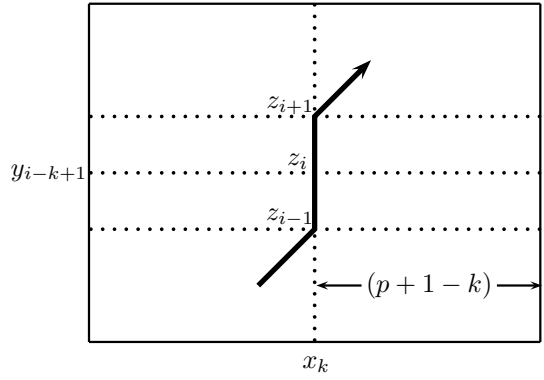

Suppose the column in question is that of $x_k$ for some $k$. This
implies that the row corresponding to $z_i$ is that indexed by
$y_{i-k+1}$. Thus in order to obtain via $\phi d_\cT$ the summand
indexed by $z_1\ldots \hat{z}_i \ldots z_{p+q}$ we must as a first step  consider
the summand of $d_\cT (\lambda \x\y)$ indexed by the term $ x_1\cdots x_py_1
\cdots \hat y_{i-k+1}\cdots y_q$. This picks up the sign
$-(-1)^{p+q+(i-k+1)}$. After applying $\phi$ the summand indexed by
$\zhat =z_1\ldots \hat{z}_i \ldots z_{p+q}$ picks up an additional 
$(-1)^{\aa(q-1)+ m(\zhat)}$ so the total sign is
\begin{equation}\label{eq:sign2}
-(-1)^{p+q+ (i-k+1) + \aa(q-1)+ m(\zhat)}.
\end{equation}
Now, $\zhat$ lives in a new grid obtained 
from the original one by deleting the $(i-k+1)$-th row. Thus
$$
m(\zhat) = m(\z) - (p+1-k),
$$
from which we deduce
$$
p+q+ (i-k+1) + \aa(q-1)+ m(\zhat) = \aa(q-1) + m(\z) + i.
$$
By checking the various possibilities for
$\aa(q)$ (or by observing that
$q+[q-1]=[q]$), we see that the signs of (\ref{eq:sign1}) and  (\ref{eq:sign2}) coincide. 
The second case, where $z_{i-1},z_{i}$ and  $z_{i+1}$ are all in the same row,
is entirely similar.
\end{proof}

As a matter of interest we note that $\phi$ is the unique chain map of this form. 
Thinking degree by degree, if $\phi$ is to be a chain map, then its definition at
$\cT_{p+q}$ is determined by its definition at $\cT_{p-1,q}$ and $\cT_{p,q-1}$,
as well as the differentials $d_\cT$ and $d_\cS$. Thus, once we have defined $\phi$ at
$\cT_{0+0}$, we have \emph{no choice\/} for the definition at $\cT_{0+1},\cT_{1+0}$,
and these in turn determine the definition at $\cT_{1+1}$, and so on, working our
way upwards and rightwards from the bottom lefthand corner of the bi-complex.
As $\cT$ and $\cS$ coincide in degree $0$, the simplest choice for
$\phi$ is the identity, and our map is precisely the chain map determined
by this choice.


\begin{proposition}\label{prop:phi}
$\phi$ is a quasi-isomorphism when
$B$ is finite and specially admissible.
\end{proposition}

\begin{proof}
Let $B$ be admissible via $x\prec 1$. 
Recalling the notation used in \S\ref{sec:les} 
there is a morphism of short exact sequences:
$$
\xymatrix{
0 \ar[r] 
&  \cT_{*}(\ov{E}(x), \FF) \ar[r]  \ar[d]_\phi
&  \cT_{*}(E, \FF) \ar[r]  \ar[d]_\phi
&  Q'_* \ar[r]  \ar[d]^{\phi^\prime} 
&0 \\
0 \ar[r]
&  \cS_{*}(\ov{E}(x), \FF) \ar[r]
&  \cS_{*}(E, \FF) \ar[r]
&  Q_*  \ar[r]   & 0.
}
$$ 
where $\phi'$ is the map induced on the quotients by the second part of
Proposition \ref{sec:main:result200}.
It is easy to check that the diagram commutes. We thus have 
a commutative diagram containing the following portion:
$$
\xymatrix{
 {\scriptstyle H_{n+1}(Q^\prime_*)} \ar[r]^\delta \ar[d]^-{\phi^\prime_*} 
& {\scriptstyle H_n(\cT_*(\ov{E}(x),\cF))} \ar[r]  \ar[d]^-{\phi_*} 
& {\scriptstyle H_n(\cT_*(E,\cF))} \ar[r] \ar[d]^-{\phi_*} 
& {\scriptstyle H_{n}(Q^\prime_*)} \ar[r]^-\delta \ar[d]^-{\phi^\prime_*} 
& {\scriptstyle H_{n-1}(\cT_*(\ov{E}(x),\cF))} \ar[d]^-{\phi_*} \\
 {\scriptstyle H_{n+1}(Q_*)} \ar[r]^\delta  
& {\scriptstyle H_n(\cS_*(\ov{E}(x),\cF))} \ar[r]   
& {\scriptstyle H_n(\cS_*(E,\cF))} \ar[r]
& {\scriptstyle H_{n}(Q_*)} \ar[r]^-\delta 
& {\scriptstyle H_{n-1}(\cS_*(\ov{E}(x),\cF))} 
}
$$
We leave it to the reader to check that the diagram
$$
\xymatrix{
  Q'_{i+j} \ar[r]^-{\psi'}  \ar[d]_{\phi'}
&  \cT_{i+j-1}(E(x), \cF)  \ar[d]_\phi \\
Q_{i+j} \ar[r]^-{\psi}
&  \cS_{i+j-1}(E(x), \cF).
}
$$
commutes, with $\psi:=\pi q$ and $\psi':=(-1)^j\pi'q'$ the quasi-isomorphisms
of \S\ref{sec:les}. In particular, we have $\phi_*'=\psi_*^{-1}\phi_*\psi'_*$.

We now argue by induction on the cardinality of the base. If $B$ is Boolean
of rank one with two elements $0<1$, then $B$ is admissible for $x=0$,
$E(x)$ is the fibre $(E_0,\cF_0)$ and $\ov{E}(x)$ is the fibre $(E_1,\cF_1)$.
Both of these are the total spaces of
a bundle over the trivial poset with a single element. 
For such a bundle $\cT_*=\cS_*$ and $\phi=\id$, inducing an isomorphism
$H_*\cT_*\rightarrow H_*\cS_*$. Thus, we have the result for $B$ using the
diagrams above and the
$5$-lemma.

In general, as $B$ is specially admissible, so are $B(x)$ and
$\ov{B}(x)$, and so the proof is again finished by induction and the $5$-lemma. 
\end{proof}

The proof of Theorem \ref{thm:main} is now also complete: Proposition 
\ref{prop:ssT} give us a spectral sequence
\[
E^2_{p,q} = H_p(B, \cH^F_q (\xi)) \Longrightarrow H_*(\cT_*(E,\cF))
\]
and Proposition \ref{prop:phi} gives us an isomorphism  $H_*(\cT_*(E,\cF))\cong  H_*(E,\cF)$.

We remark
here that our proof relies heavily on having a specially admissible base because in
general the long exact sequences used in the previous proposition do not exist. 
Whether the theorem
remains true in greater generality remains an open question.


\begin{example} 
The bundle perspective and the spectral sequence above provide an alternative 
proof of invariance under the first Reidemeister
move for Khovanov homology. Let $D$ and $D^\prime$ be link
diagrams, identical except in a small disc as on the top left of Figure 
\ref{figure520}.

\begin{figure}
  \centering
\begin{pspicture}(12.5,6)
\rput(-1.25,0){
\rput(1,0){
\rput(2,5){\BoxedEPSF{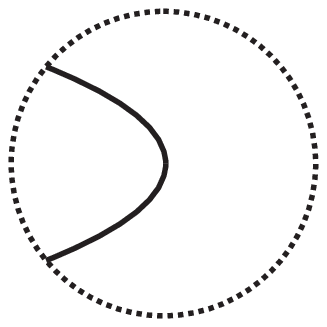 scaled 670}}
\rput(5,5){\BoxedEPSF{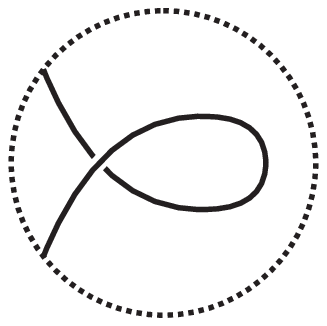 scaled 670}}
\rput(3.5,2){\BoxedEPSF{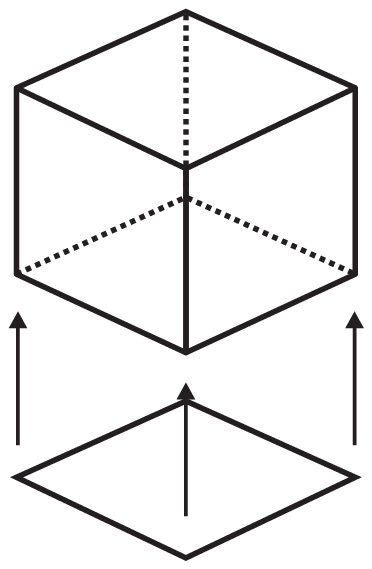 scaled 670}}
\rput(10,1.5){\BoxedEPSF{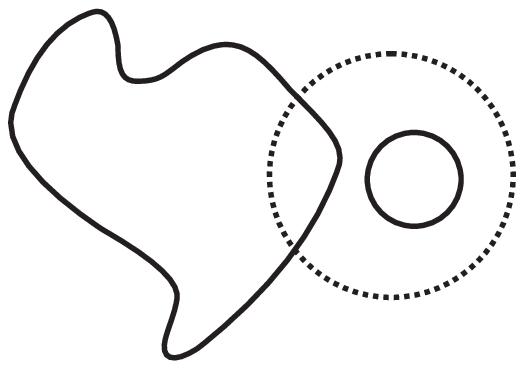 scaled 670}}
\rput(10,4.5){\BoxedEPSF{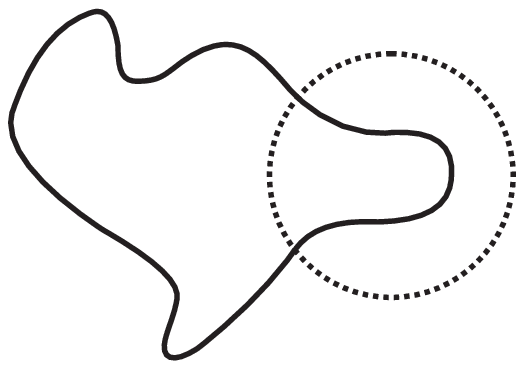 scaled 670}}
\rput(.5,5){$D$}\rput(6.5,5){$D'$}
\rput(1.75,.75){$\bB_N$}\rput(1.75,2.7){$\bB_{N+1}$}
\psline[linewidth=.3mm]{->}(11,2.6)(11,3.6)
\rput(11.3,3.05){$m$}
\psframe[linewidth=1mm,linecolor=gray](4.4,1.9)(4.9,3.5)
\psline[linewidth=1mm,linearc=.3,linecolor=gray]{-}(8,3)(8,6)(12,6)(12,0)(8,0)(8,3)
\psline[linewidth=1mm,linearc=.3,linecolor=gray]{<-}(4.9,2.7)(7,2.7)(7,3.5)(8,3.5)
}}
\end{pspicture}
  \caption{Decomposing a Khovanov-coloured Boolean lattice $\bB_{N+1}$ as 
a bundle over $\bB_N$ to show invariance of the first Reidemeister move.}
  \label{figure520}
\end{figure}

For a link diagram $D$ with $N$ crossings, let $\FF_D$ be the Khovanov colouring
of the Boolean lattice $\bB_N$ on the crossings.
To show invariance under
Reidemeister I, we require $H_{*+1}(\bB_{N+1},\FF_{D'})\cong H_{*}(\bB_N,\FF_D)$.

Consider the crossings of $D'$ \emph{outside\/} the disk. We can decompose
$(\bB_{N+1},\FF_{D'})$ as a bundle over the Boolean lattice $\bB_N$ on these crossings:
the fibres correspond to the two possible resolutions of the crossing \emph{in\/} the disk
as in Figure \ref{figure520} (with $n=2$). 

In particular, the complex $\cC_*$ for the fibres is 
$\id^{\otimes i}\otimes m\otimes\id^{\otimes j}\colon V^{\otimes i+j+2} \ra V^{\otimes i+j+1}$
where $V$ is the Frobenius algebra used in the definition of Khovanov
homology and $m$ is its multiplication. As $m$ is onto and 
$\ker(m)\cong V$ we have $\cH_*^{\text{fib}}(\xi)$ is $V$ in degree $1$ and $0$ elsewhere,
thus the homology colouring $(\B_N,\cH_1^{\text{fib}}(\xi))$ of the base
is $(\B_N,\FF_D)$. The spectral sequence thus has $E^2$-page
$$
E^2_{p,q} = \begin{cases}
            H_p(\bB_N, \cF_D), & q=1\\
	    0, & \text{otherwise.}
	    \end{cases}
$$
collapsing, and so $H_{*-1}(\B_N,\FF_D)\cong H_*(\B_{N+1},\FF_{D'})$. Now
apply the main theorem of \cite{Everitt-Turner}, turning this isomorphism
into the required one between the Khovanov homologies.
\end{example}




\section{An application to knot homology}
\label{sec:appl}

In this section we give an application of the spectral sequence
of Theorem \ref{thm:main}
to knot homology. We will write the results in terms of Khovanov
homology, but in fact one can be more general--see the remarks at the
end of the section.

The idea is the following. Starting with a link diagram we choose a
subset of the crossings to be fixed and form a cube of link diagrams by
resolving the remaining crossings in the way familiar in Khovanov
homology (see Khovanov \cite{Khovanov00} and Bar-Natan
\cite{Bar-Natan02} ). We are now at liberty to take the Khovanov
homology of each link diagram at a cube vertex and moreover to
consider the induced maps associated to Morse moves along the
edges. From this we can form a complex \`a la Khovanov. The main
result is that there is a spectral sequence converging to the Khovanov
homology of the original diagram with the homology of this new complex
at the $E^2$-page.

Our conventions for shifting complexes will be that for a
tri-graded vector space $W_{*,*,*}$ we define $W_{*,*,*}[a,b,c]$ by
$(W_{*,*,*}[a,b,c])_{i,j,k}= W_{i-a,j-b,k-c}$. Similar conventions
apply to bi-graded and singly graded vector spaces.

First we outline our grading and normalisation conventions in
constructing Khovanov homology. Let $D$ be a link diagram with $N$
crossings. The {\em unnormalised} Khovanov homology of the diagram $D$,
denoted $\obkh ** (D)$, is defined to be the homology of the complex $\obc ** (D)$ given as
$$
\obc i * (D) =\kern-3mm\bigoplus_{\rk(\alpha) = N-i}
\kern-3mmV^{\otimes k_\alpha}[\text{rk}(\alpha)]
$$ where $V$ is the usual graded vector space defining Khovanov
homology, $k_\alpha$ is the number of circles appearing in
smoothing $\alpha$ and $\rk$ is the rank function on the 
Boolean lattice. The differential $d\colon \obc i * (d) \ra \obc
{i-1} *$ is the usual one.

\begin{figure}
  \centering
\begin{pspicture}(12.5,5.75)
\rput(0,0){
\rput(3,2.5){\BoxedEPSF{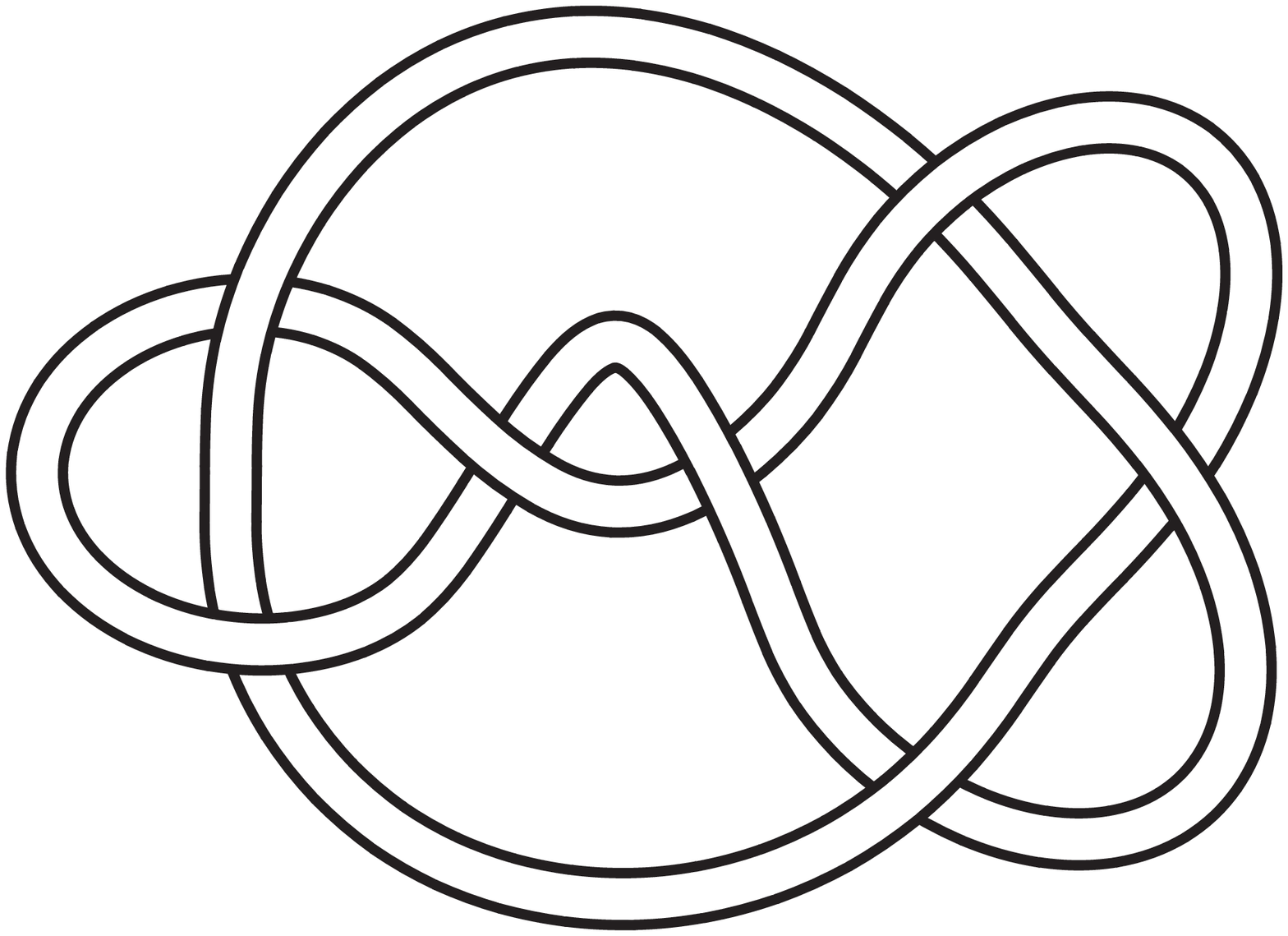 scaled 150}}
\pscircle[linewidth=.4mm](2.15,2.15){.2}
\pscircle[linewidth=.4mm](2.7,2.525){.2}
\pscircle[linewidth=.4mm](3.65,3.1){.2}
\pscircle[linewidth=.4mm](3.65,1.8){.2}
}
\rput(0,-.5){
\rput(9,3.25){\BoxedEPSF{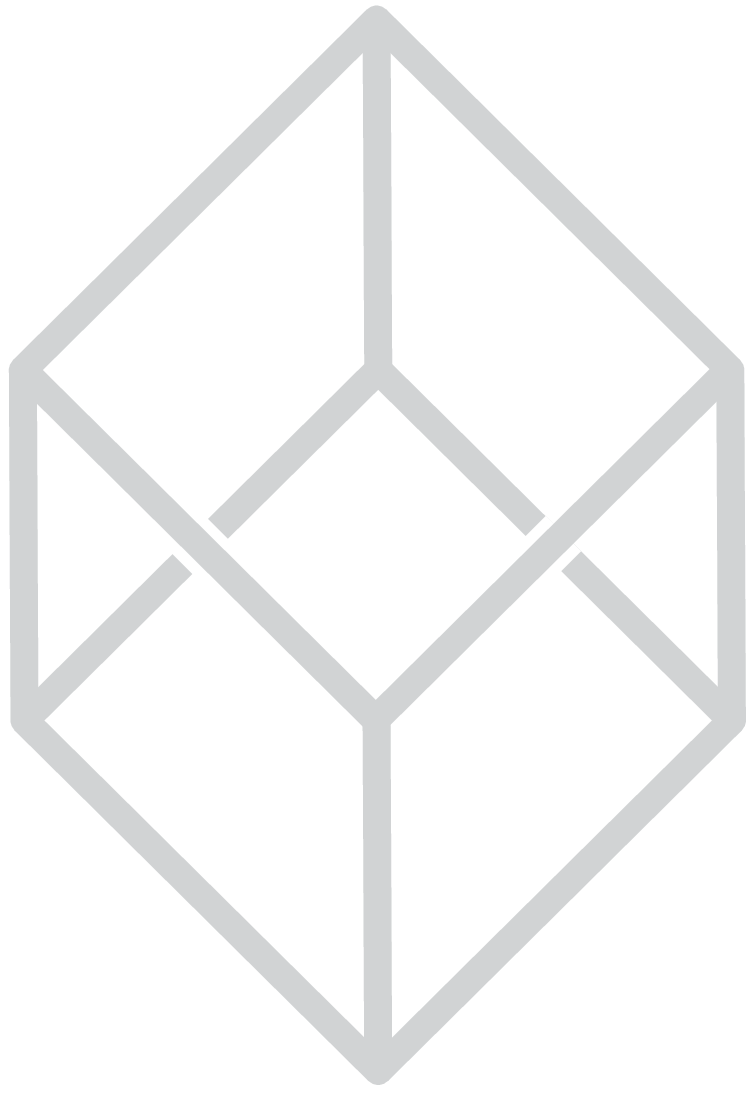 scaled 500}}
\rput(9,5.75){\BoxedEPSF{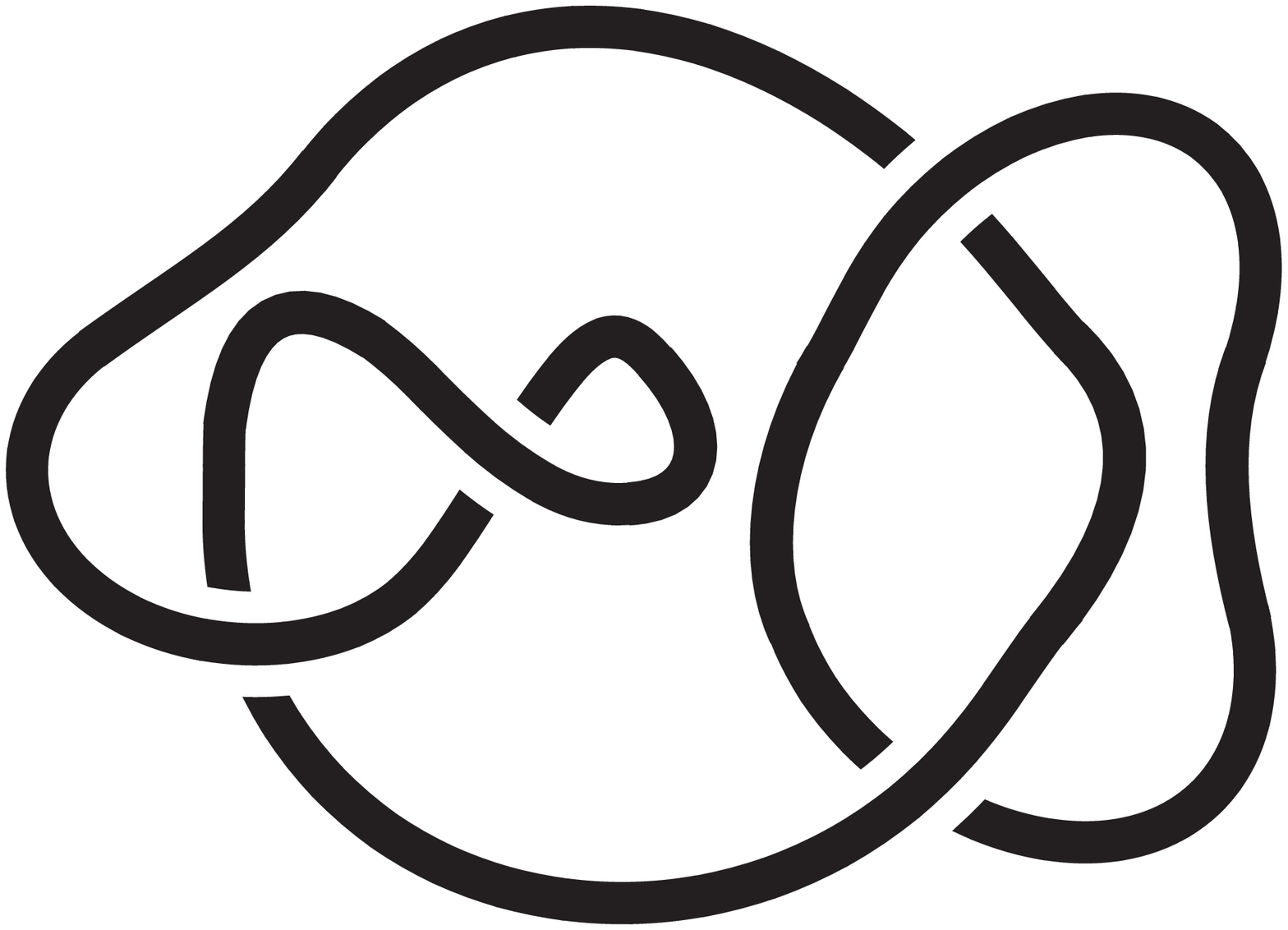 scaled 60}}
\rput(7.2,4){\BoxedEPSF{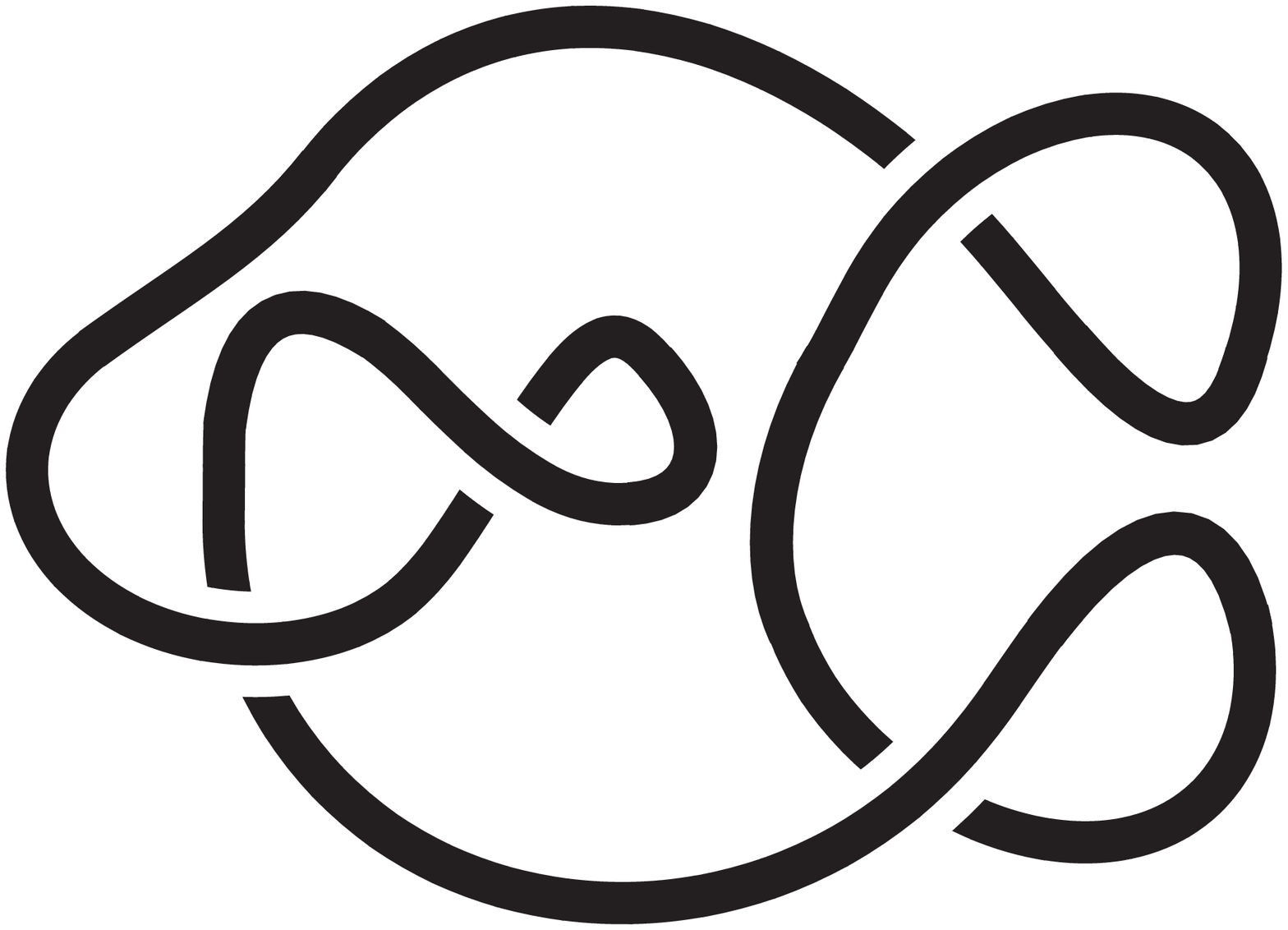 scaled 60}}
\rput(9,4){\BoxedEPSF{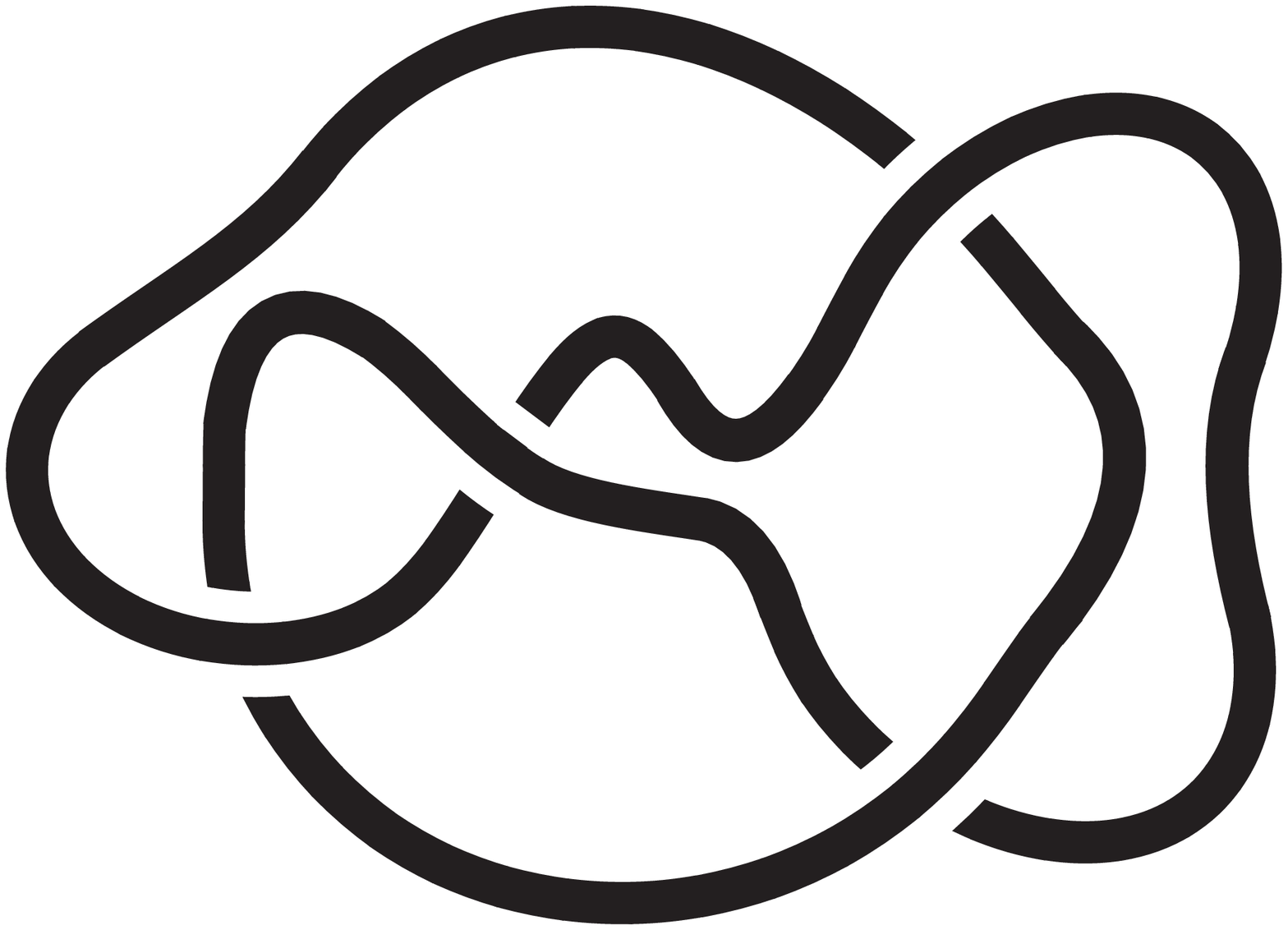 scaled 60}}
\rput(10.8,4){\BoxedEPSF{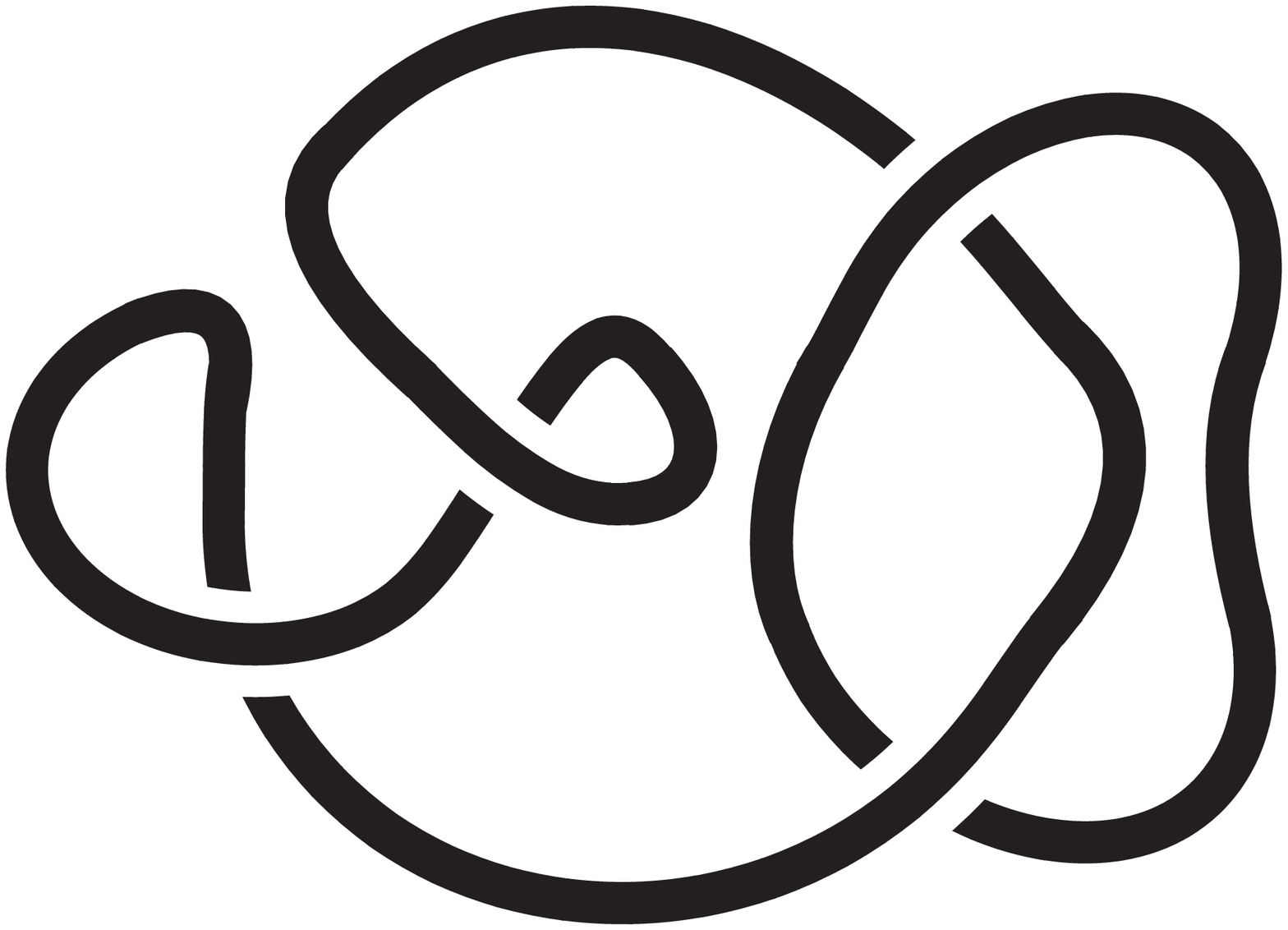 scaled 60}}
\rput(7.2,2.35){\BoxedEPSF{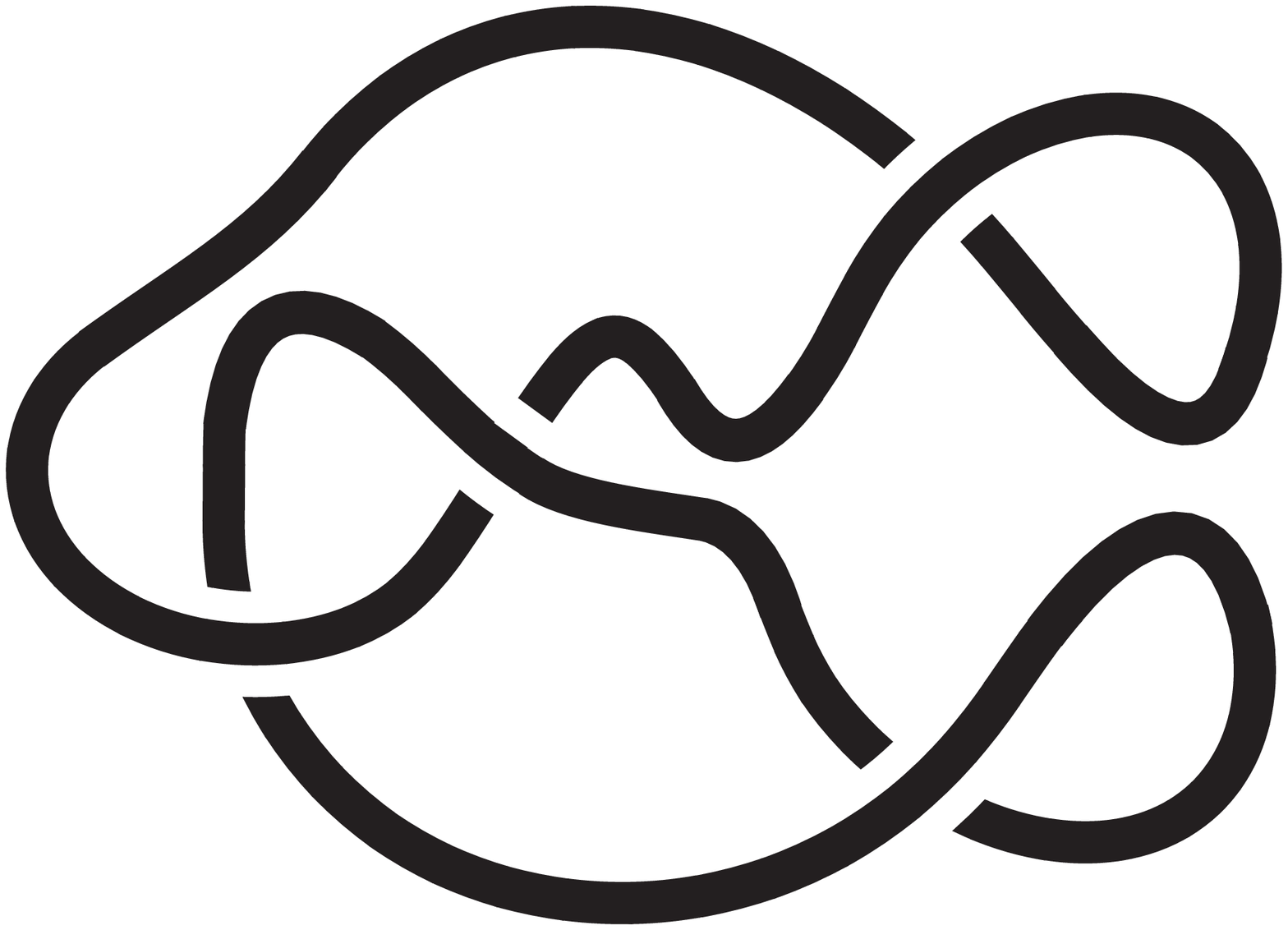 scaled 60}}
\rput(9,2.35){\BoxedEPSF{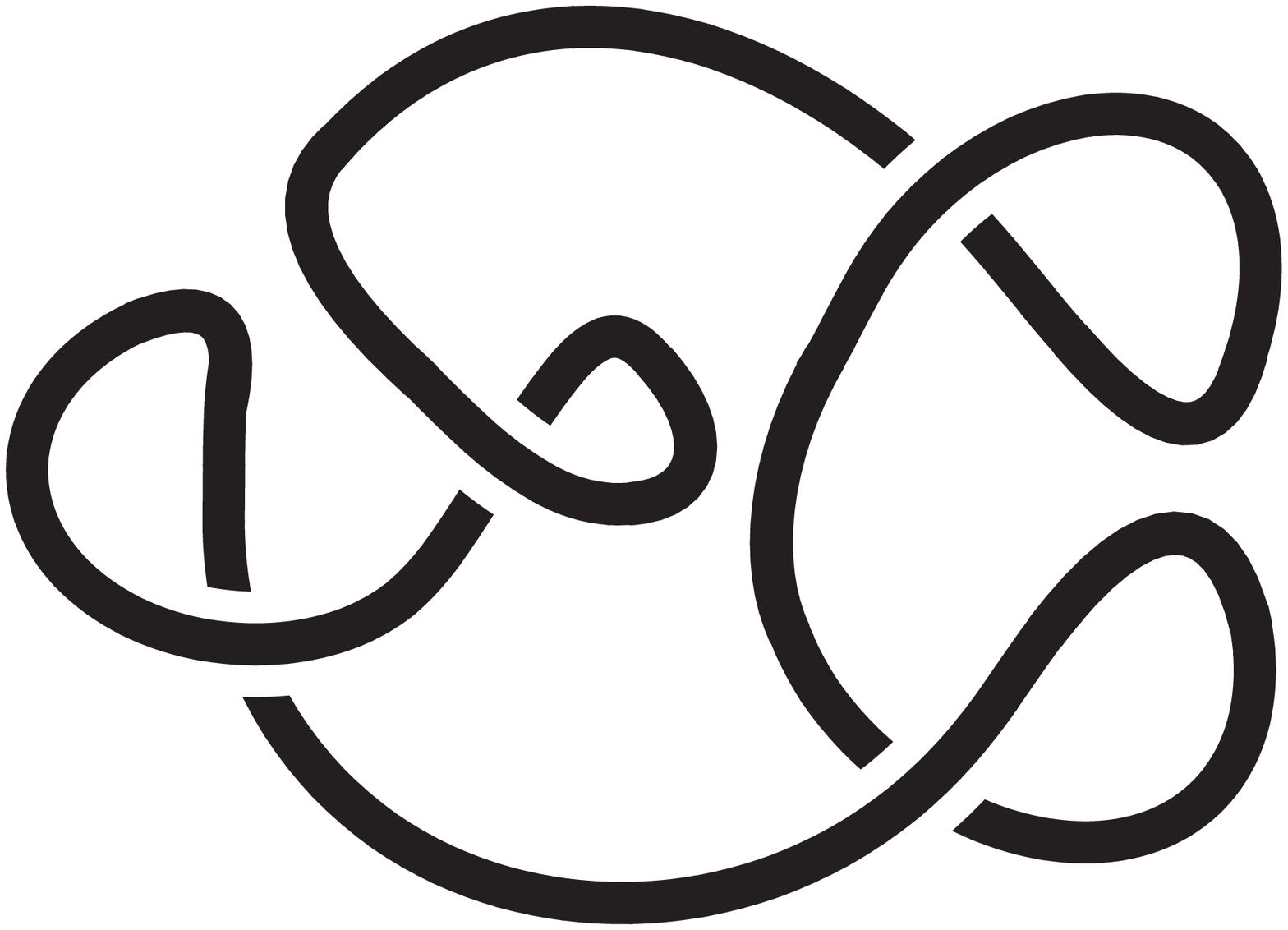 scaled 60}}
\rput(10.8,2.35){\BoxedEPSF{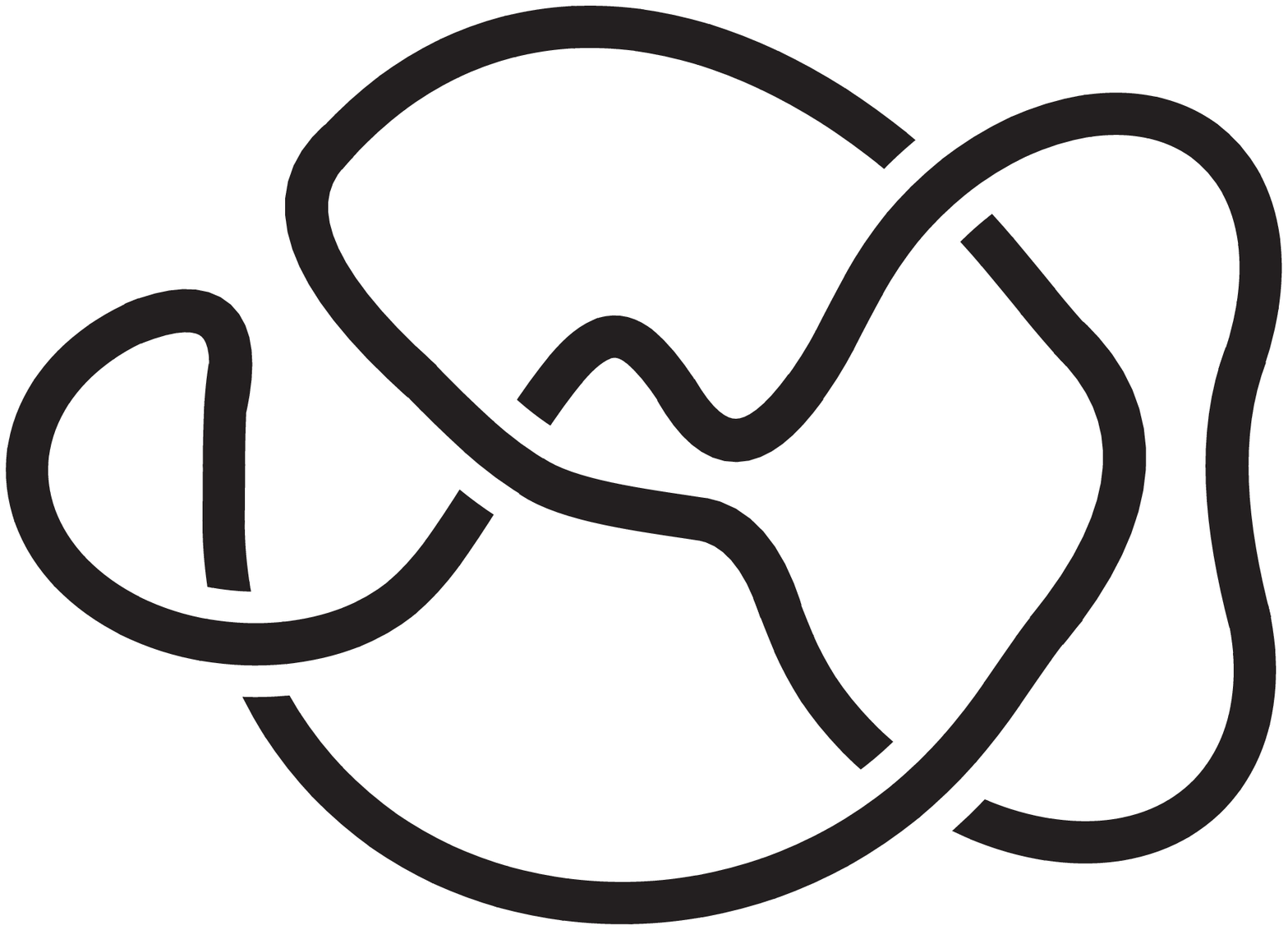 scaled 60}}
\rput(9,.8){\BoxedEPSF{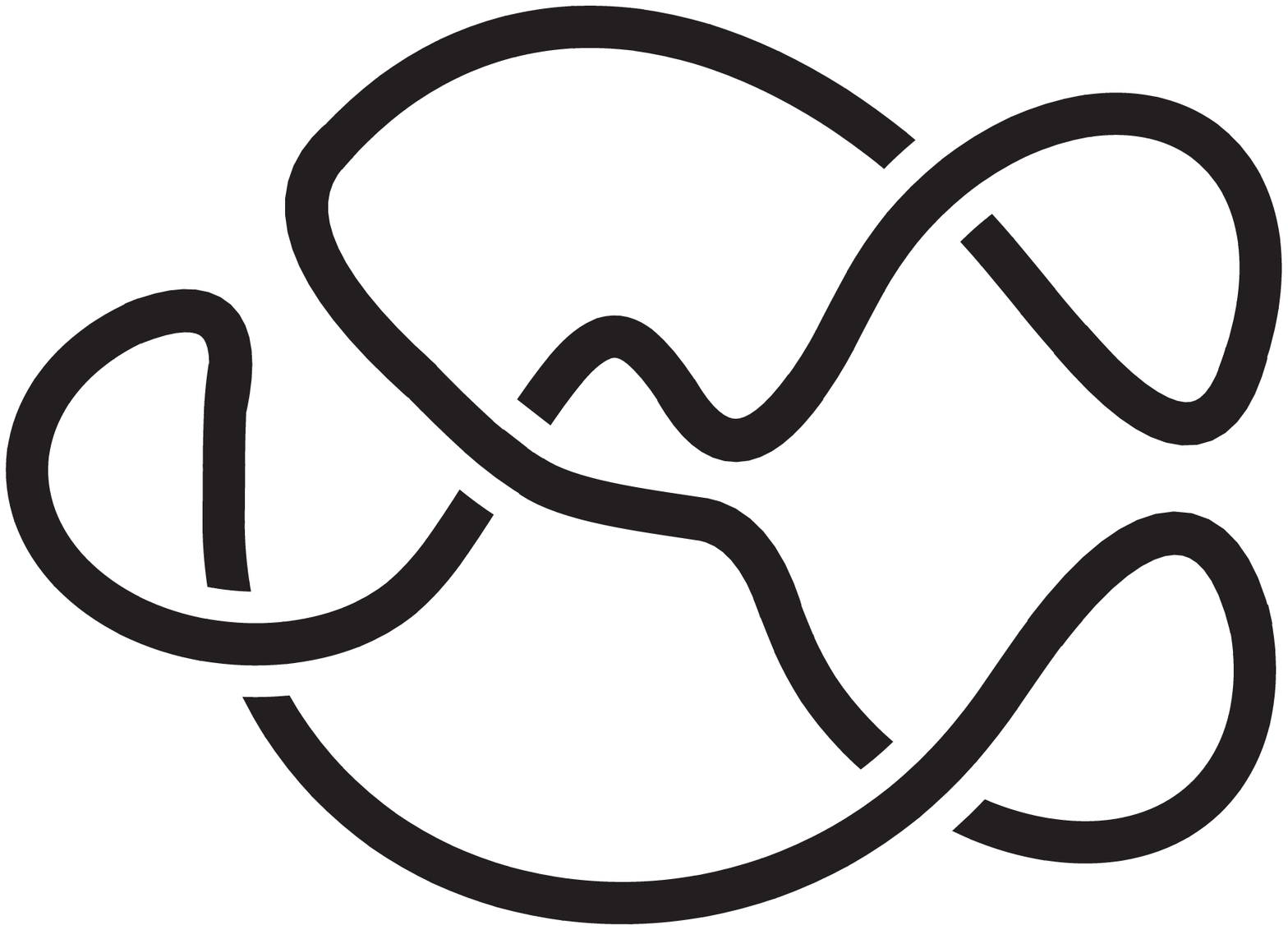 scaled 60}}
}
\end{pspicture}
  \caption{the knot $7_5$ with four fixed crossings}
  \label{figure600}
\end{figure}

If the diagram $D$ is oriented, with $N_+$ positive 
and $N_-$ negative crossings, we can also define the {\em normalised} 
Khovanov homology to be,
$$
\kh ** (D) = \obkh ** (D)[-N_+, N_+ - 2N_-].
$$ 
The reader should note that this differs from the original
conventions followed by Khovanov \cite{Khovanov00} and Bar-Natan
\cite{Bar-Natan02}, where unnormalised homology is defined via a
cochain complex $\obcoriginal ** (D)$ which is normalised to define $
\khoriginal ** (D)$, the latter being an invariant of oriented
links. The conventions we use are related to the original ones by
$$
\obcoriginal i j (D) = \obc {N-i} j (D),
$$
$$
\khoriginal i j (D) = \kh {-i} j (D).
$$

Now to the complex referred to above. 
Let $D$ be an oriented link diagram with $N$ crossings. Choose $k$
crossings $c_1, \ldots , c_k$ and call them the {\em fixed
crossings}. The remaining $\ell=N-k$ crossings are the
{\em free crossings}, and we assume these too have been ordered.

Resolving the free crossings in one of the two familiar
ways yields $2^\ell$ diagrams which we place at
the vertices of the Boolean lattice $\bB_\ell$ (a {\em cube} in more
usual Khovanovology). Figure \ref{figure600} illustrates this for the knot $7_5$ 
with the four fixed crossings as shown in red.

For $x\in\bB_\ell$ let $D_x$ denote the corresponding link diagram. Note
that $D_x$ does not in general inherit an orientation from $D$, and we will be 
treating $D_x$ as an unoriented diagram. 

We now associate to each vertex $x$ of the cube the bigraded module
$$V_{*,*}(x)= \obkh ** (D_x)[0,\text{rk}(x)]$$ 
the unnormalised Khovanov
homology of the (unoriented) diagram $D_x$ shifted in the second degree by the 
rank of $x$  in $\bB_\ell$. 

If $x<x^\prime\in\cB_\ell$ then the diagrams $D_x$ and $D_{x^\prime}$ are
identical except in a small disc within which one of the zero
smoothings in $D_x$ turns in to a 1-smoothing in $D_{x^\prime}$.  By
using the multiplication or comultiplication in the familiar way
(geometrically using the saddle cobordism) there is a chain map,
$$
\obc ** (D_x) \ra  \obc ** (D_{x^\prime}).
$$ 
This map has bi-degree $(0,-1)$. Thus, after taking homology and shifting,  
we have a homomorphism of bidegree $(0,0)$,
$$
\delta_x^{x^\prime} \colon V_{*,*}(x) \ra  V_{*,*}({x^\prime}).
$$

We now define the complex $\obcK_{*,*,*}(D;c_1, \ldots , c_k)$ by setting
$$
\obcK_{i,*,*}(D;c_1, \ldots , c_k) =\kern-5mm\bigoplus_{x\in \bB_\ell, \text{rk}(x) = N-i}
\kern-5mmV_{*,*}(x).
$$ 
The differential is defined using the maps $\delta_x^{x^\prime}$ in
the usual way (and using the usual signage conventions in Khovanov
homology). It has tri-degree $(1,0,0)$,
and is indeed a differential because  all squares in the cube
anti-commute (even though these are now Morse-move frames in a movie of link diagrams) 
because saddles can be re-arranged without any unwanted side effects.


We now define the Khovanov homology of the diagram $D$ \emph{with respect to the given 
fixed crossings\/} to be,
$$
\ob{K\!H}_{*,*,*}(D; c_1, \ldots , c_k):= H(\obcK_{*,*,*}(D; c_1, \ldots , c_k)).
$$
It is important to note that $\ob{K\!H}_{*,*,*}(D; c_1, \ldots , c_k) $ depends on the set 
of fixed crossings and is with respect to a specific diagram. 
However, $\ob{ K\!H}_{*,*,*}(D; c_1, \ldots , c_k) $ is invariant under re-ordering of the 
fixed and free crossings. If the set of crossings is empty then we recover the 
unormalised Khovanov homology.

With these preliminaries let us now state the theorem.

\begin{theorem}\label{thm:mainappl}
Let $D$ be an $N$-crossing link diagram and let $c_1, \ldots , c_k$ be $k$ crossings of $D$. 
For each $i$, there is a spectral sequence
\[
E^2_{p,q} = \ob{K\!H}_{p,q,i}(D; c_1, \ldots , c_k)  \Longrightarrow \obkh {p+q} i (D).
\]
The $r$-th differential in the spectral sequence has bidegree $(-r,r-1)$.
\end{theorem}

\begin{proof}
The procedure for constructing $\obkh ** (D)$ includes (in our language) 
the Khovanov colouring $\cF$ of the 
Boolean lattice $\B_{k+\ell}$, on the set of 
crossings of the diagram. It is a graded colouring in that it takes values 
in graded modules. 

We have been given a subset of the crossings in the form of the
fixed crossings. By Example \ref{ex:boolean} we can decompose
$(\B_{k+\ell}, \cF)$ as a bundle $\xi$ over the Boolean
lattice on the free crossings, i.e. a copy of $\B_\ell$. Note that the total space of this
bundle is $(\B_{k+\ell},\cF)$.

Fixing $i$ and applying Theorem \ref{thm:main} we thus have a spectral sequence
\[
E^2_{p,q} = H_p(\B_\ell, \cH^{\text{fib}}_q (\xi)) \Longrightarrow H_{*,i}(\B_{k+\ell},\cF),
\]
where the colouring $\cF$ is the restriction of the Khovanov colouring to grading $i$.

Recall from \cite{Everitt-Turner} that
\[
\obkh ** (D) \cong H_{*,*} (\B_{k+\ell},\cF),
\]
Thus the spectral sequence above converges as required to $ \obkh * i (D)$,
and it remains to identify the $E^2$-page.
 
It is clear that for $x\in \B_\ell$ the fibre over $x$ is isomorphic to
$\B_k$ equipped with the Khovanov colouring $\cF_x$ of the diagram
$D_x$ shifted by $\text{rk}(x)$. Thus the homology of the fibre over
$x$ is given by $ H_{*,*} (\B_{k},\cF_x)[0,\text{rk}(x)]$. Using again
the isomorphism between Khovanov homology and coloured poset homology,
we see that the homology of the fibre is isomorphic to $ \obkh **
(D_x)[0,\text{rk}(x)]= V_{*,*}(x)$. The induced maps between fibres
are by definition the maps $\delta_x^{x^\prime}$. Thus we have
coloured $\B_\ell$ by the modules $V_{*,*}(x)$ and the maps
$\delta_x^{x^\prime}$. The $E^2$-page of the spectral sequence is the
coloured poset homology of this coloured poset and
$\ob{K\!H}_{*,*,i}(D; c_1, \ldots , c_k)$ is the Khovanov
homology of it. Appealing again to \cite{Everitt-Turner}, these two homologies 
agree, and hence the result is proved.
\end{proof}

\begin{remark}
In practice it may well be useful to apply Reidemeister moves to $D_x$ and a 
quick look at Figure \ref{figure600} should be enough to see why. 
Since we are using unnormalised Khovanov homology to colour
fibres we must be a little careful as shifts need to be introduced. The
best way to proceed is to choose an orientation for $D_x$ and count
the number of positive and negative crossings $n_+$ and $n_-$. One can
then express $V_{*,*}(x)$ in terms of normalised Khovanov homology as
$$
V_{*,*}(x) = \kh **  ({D}_x)[n_+, 2n_- - n_+ + \text{rk}(x)],
$$ 
and now of course we are free to compute $\kh **  ({D}_x)$ by 
applying Reidemeister moves if necessary. 
\end{remark}

\begin{remark}
All of the above works in more generality than is stated. Indeed one
can do a similar thing given a link homology theory which can be
extended to knotted resolutions. An important example of this is
Khovanov-Rozansky theory which is constructed from a cube of
resolutions where each resolution can be viewed as a 4-valent
graph. By fixing some subset of crossings we get a (smaller) cube of
knotted 4-valent graphs. Khovanov-Rozansky homology has been extended
to such knotted graphs by Wagner in \cite{Wagner}. By following
through the procedure above one then gets a spectral sequence
converging to Khovanov-Rozansky homology with a diagram-dependent, but
in principle calculable, $E^2$-page.
\end{remark}




\bibliographystyle{amsplain}

\end{document}